\documentclass[11pt,reqno]{amsart}
\usepackage{a4wide}
\usepackage[
backend=biber,
style=alphabetic,
sorting=nyt,
maxbibnames=50,
giveninits=true,
maxalphanames=10
]{biblatex}
\usepackage{csquotes}
\renewbibmacro{in:}{%
 \ifentrytype{article}{}{}}
\DeclareFieldFormat[article]{volume}{\textbf{#1}}
\DeclareFieldFormat[article]{title}{\textit{#1}}
\DeclareFieldFormat[article]{journaltitle}{#1}
\DeclareFieldFormat[article]{number}{\textbf{#1}}
\renewbibmacro*{volume+number+eid}{%
  \printfield{volume}%
\setunit*{\textbf{\addcolon}}
  \printfield{number}\setunit{\addcomma\space}
  \printfield{eid}}

\usepackage{verbatim}
\usepackage[shortlabels]{enumitem}
\usepackage[T1]{fontenc}
\usepackage[unicode,hypertexnames=false,colorlinks=true,linkcolor=blue,citecolor=blue]{hyperref}
\usepackage{esint}
\usepackage{hyperref}
\usepackage[capitalise, nameinlink, noabbrev]{cleveref} 

\hypersetup{colorlinks,breaklinks,
	linkcolor=[rgb]{0,0,0},
	citecolor=[rgb]{0,0,0},
	urlcolor=[rgb]{0,0,0}}

\newcommand{\be}{\begin{equation}}
	\newcommand{\ee}{\end{equation}}
\theoremstyle{plain}

\newtheorem{theorem}{Theorem}[section]
\newtheorem{proposition}[theorem]{Proposition}
\newtheorem{corollary}[theorem]{Corollary}
\newtheorem{lemma}[theorem]{Lemma}
\newtheorem{definition}[theorem]{Definition}

\newtheorem{remark}[theorem]{Remark}

\newcommand{\EE}{\mathcal{E}(\lambda,\Lambda)}

\newcommand{\R}{\mathbb{R}}

\newcommand{\N}{\mathbb{N}}
\newcommand{\LL}{\mathcal{L}}
\newcommand{\eps}{\varepsilon}
\newcommand{\F}{\mathcal{F}}

\newcommand{\bea}{\begin{equation*}\begin{aligned}}
		\newcommand{\eea}{\end{aligned}\end{equation*}}

\makeatletter
\newcommand*{\wackyenum}[1]{%
  \expandafter\@wackyenum\csname c@#1\endcsname%
}

\newcommand*{\@wackyenum}[1]{%
  $\ifcase#1\or(H_\F)\or(H_f)\or(H_g)\or(H_p)%
    \else\@ctrerr\fi$%
}
\AddEnumerateCounter{\wackyenum}{\@wackyenum}{53.13}
\makeatother

\crefname{subsection}{subsection}{subsections}

\numberwithin{equation}{section}

\numberwithin{equation}{section}

\title[Existence and regularity in the fully nonlinear one-phase problem]{Existence and regularity in the fully nonlinear one-phase free boundary problem}

\author[M. Carducci]{Matteo Carducci}\thanks{}
\address {Matteo Carducci \newline \indent
	Classe di Scienze, Scuola Normale Superiore \newline \indent
	Piazza dei Cavalieri 7, 56126 Pisa - ITALY}
\email{\href{mailto:matteo.carducci@sns.it}{matteo.carducci@sns.it}}

\author[B. Velichkov]{Bozhidar Velichkov}\thanks{}
\address {Bozhidar Velichkov \newline \indent
	Dipartimento di Matematica, Universit\`a di Pisa \newline \indent
	Largo B. Pontecorvo 5, 56127 Pisa - ITALY}
\email{\href{mailto:bozhidar.velichkov@unipi.it}{bozhidar.velichkov@unipi.it}}
\begin{document}
	
	\subjclass[2010] {
	}
		\begin{abstract} We consider viscosity solution to one-phase free boundary problems for general fully nonlinear operators and free boundary condition depending on the normal vector. We show existence of viscosity solutions via the Perron's method and we prove $C^{2,\alpha}$ regularity of flat free boundaries via a quadratic improvement of flatness. Finally, we obtain the higher regularity of the free boundary via an hodograph transform.
    	\end{abstract}
	
	\maketitle
        \tableofcontents
	
	\section{Introduction}	
In this paper we study existence and regularity of the free boundary of viscosity solutions to one-phase free boundary problems for fully nonlinear equations with right-hand side and with general free boundary condition. 
Precisely, we consider $u:B_1\subset\R^d\to\R$ non-negative continuous viscosity solutions to the following problem
	\begin{equation}\label{def:def-viscosity-solution}\begin{cases}
			\F(D^2u,\nabla u,x)=f(x) \quad&\text{in } \Omega_u\cap B_1,\\
			|\nabla u|=g(x,\nu) \quad&\text{on } \partial\Omega_u\cap B_1,\\
		\end{cases}
	\end{equation}
	where  
    $\Omega_u:=\{u>0\}$ is the positivity set of $u$ and $\nu\in\mathbb{S}^{d-1}$ is the unit normal to $\partial\Omega_u$ pointing inwards.\smallskip

	We consider fully nonlinear operators $\F:\mathcal{S}^{d\times d}\times \R^d\times B_1\to\R$ which are uniformly elliptic, that is, we assume that there are constants $0<\lambda\le\Lambda$ and $[\F]_{\text{Lip}(\R^d)}>0$ such that 
    \be\label{eq:pucci}\mathcal{M}^-(N)-[\F]_{\text{Lip}(\R^d)}|\xi-\eta|\le \F(M+N,\xi,x)-\F(M,\eta,x)\le \mathcal{M}^+(N)+[\F]_{\text{Lip}(\R^d)}|\xi-\eta|,\ee 
    for every $M,N\in\mathcal{S}^{d\times d}$, $\xi,\eta\in \R^d$, $x\in B_1$, where $\mathcal S^{d\times d}\subset \R^{d\times d}$ is the space of real $d\times d$ symmetric matrices, and where $\mathcal{M}^\pm:\mathcal{S}^{d\times d}\to \R$ are the Pucci's extremal operators defined, for any $N\in\mathcal S^{d\times d}$, as \bea\mathcal{M}^+(N):=\Lambda \sum_{\mu_i>0}\mu_i+\lambda \sum_{\mu_i<0}\mu_i\quad\text{and}\quad \mathcal{M}^-(N):=\lambda \sum_{\mu_i>0}\mu_i+\Lambda \sum_{\mu_i<0}\mu_i,\eea where $\mu_i$ are the eigenvalues of $N$ (see e.g. \cite{caffarellicabre}). We also suppose that $\F(0,0,x)\equiv0$, and for simplicity, we use the notation
    \bea
    \mathcal{E}(\lambda,\Lambda):=\{\F:\mathcal{S}^{d\times d}\times \R^d\times B_1\to\R:\text{ $\F$ satisfies \eqref{eq:pucci} and $\F(0,0,x)\equiv0$}\}.
    \eea
	The minimum hypotheses that we assume are $$\F\in\EE,\quad  f\in C^0( B_1),\quad g\in C^0( B_1,\mathbb{S}^{d-1})$$ 
    and we will ask for more regularity, where necessary. Here below is a list of hypotheses, on the fully nonlinear operator $\F$ and the free boundary condition $g$, that we will use (together or separately) throughout the paper.
	\begin{enumerate}[label=(H\arabic*)]
    \item\label{hyp:h1} The free boundary condition $g$ satisfies the lower bound $$\gamma_0:=\inf_{(x,\nu)\in B_1\times \mathbb{S}^{d-1}}g(x,\nu)>0.$$ 
    
        \item\label{hyp:h2} The operator 
        $$\F(\cdot,\xi,x)\text{ is either concave or convex for every }\xi\in\R^d,\ x\in B_1.$$ 
		
		\item\label{hyp:h3} The operator $\F$ is $C^{0,\beta}$ in the third variable, namely $$|\F(M,\xi,x)-\F(M,\xi,y)|\le [\F]_{C^{0,\beta}(B_1)}(1+|\xi|+|M|)|x-y|^{\beta},$$ for every $M\in\mathcal{S}^{d\times d}$, $\xi\in\R^d,$ $ x,y\in  B_1,$ for some $\beta\in(0,1]$ and $[\F]_{C^{0,\beta}(B_1)}>0$.
	\end{enumerate}   
\smallskip
We use the following notion of viscosity solution to the free boundary problem \eqref{def:def-viscosity-solution}.
    
\begin{definition}[Viscosity solutions to \eqref{def:def-viscosity-solution}]\label{def:def-sol}
	Suppose that $\F\in\EE$, $f\in C^0(B_1)$, $g\in C^0(B_1,\mathbb{S}^{d-1})$ and let $u:B_1\to\R$ be a non-negative continuous function. We say that $u$ is a viscosity subsolution (resp. supersolution) to \eqref{def:def-viscosity-solution}, if $u$ satisfies 
        $$\F(D^2u,\nabla u,x)\ge f\quad\text{in }\Omega_u\cap B_1\qquad\text{(resp. $\le$)}$$  
        in viscosity sense and if $u$ satisfies the free boundary condition in the following sense.
	Let	$x_0\in \partial\Omega_u\cap B_1$ and $\varphi\in C^1(B_1)$, such that $\varphi_+$ touches $u$ from above (resp. $\varphi$ touches $u$ from below) at $x_0$, then $$0\not=|\nabla \varphi(x_0)|\ge g\left(x,\frac{\nabla\varphi(x_0)}{|\nabla\varphi(x_0)|}\right)\qquad\text{(resp. $\le$)}.$$  
		We say that $u$ is a viscosity solution to \eqref{def:def-viscosity-solution} if it is both a viscosity subsolution and a viscosity supersolution to \eqref{def:def-viscosity-solution}.
	\end{definition}

\subsection{State of the art}
The first existence and regularity result for the one-phase problem \eqref{def:def-viscosity-solution}, with Laplace operator and zero right-hand side, was proved by Alt and Caffarelli in \cite{AltCaffarelli:OnePhaseFreeBd} by a variational approach. The construction of viscosity solutions to free problems via the Perron's method was initiated by Caffarelli in \cite{Caffarelli1988:HarnackApproachFreeBd3Existence} in the context of the two-phase problem for linear operators with zero right-hand side (see also \cite{CaffarelliSalsa:GeomApproachToFreeBoundary}), and was extended to fully nonlinear homogeneous and concave operators of the form $\mathcal F=\mathcal F(D^2u)$ by Wang \cite{wangIII}. Again for the two-phase problem, in \cite{DESILVA-perron}, De Silva, Ferrari and Salsa proved existence of viscosity solutions in the case of linear operators with non-zero right-hand side, while Salsa, Tulone and Verzini \cite{salsatuloneverzini-existenceviscositysol} generalized this result to the case of fully nonlinear homogeneous and concave operators of the form $\mathcal F=\mathcal F(D^2u)$. \smallskip

For viscosity solutions, the $C^{1,\alpha}$ regularity of the flat free boundaries of the two-phase problem, for the Laplacian with zero right-hand side, was first obtained by Caffarelli in the papers \cite{Caffarelli1987:HarnackApproachFreeBd1LipschitzC1a,Caffarelli1989:HarnackApproachFreeBd2FlatLipschitz} (see also \cite{CaffarelliSalsa:GeomApproachToFreeBoundary}); further $C^{1,\alpha}$ regularity results for viscosity solutions to the two-phase problem, with linear or fully nonlinear operators and zero right-hand side, were obtained in \cite{wangI,wangII,fel01,ceruttiferrarisalsa-linearoperators,fe1,ferrarisalsa-linearoperators,fs2,af}. In \cite{DeSilva:FreeBdRegularityOnePhase} De Silva introduced a different improvement of flatness technique, and proved a $C^{1,\alpha}$ regularity result for the one-phase problem for linear operators with right-hand side; for the two-phase problem, the same technique was used in \cite{desilvaferrarisalsadistributed,desilvaferrarisalsadivergence} in the context of linear operators with right-hand side, and in \cite{desilvaferrarisalsafullynonlinear} for a class of fully nonlinear operators; we also refer to \cite{rampassoetall} for the case of (degenerate) one-phase fully nonlinear problems.\smallskip 

For what concerns the higher regularity of the free boundaries, Kinderlehrer and Nirenberg \cite{KinderlehrerNirenberg1977:AnalyticFreeBd} proved that the $C^{2,\alpha}$ free boundaries are $C^\infty$ by an hodograph transform approach. In \cite{KriventsovLin}, it was shown that in the case of divergence form operators, the $C^{1,\alpha}$ regularity is sufficient to apply the technique from \cite{KinderlehrerNirenberg1977:AnalyticFreeBd} and to obtain (in the case of smooth data) the $C^\infty$ regularity. Shortly afterwards, in \cite{dfs-higher-order-two-phase}, using a different approach based on an improvement of flatness argument, De Silva, Ferrari and Salsa proved that the $C^{1,\alpha}$ free boundaries are $C^{2,\alpha}$ for the two-phase problem for the Laplacian with right-hand side; this implies the $C^\infty$ regularity of the free boundaries by \cite{KinderlehrerNirenberg1977:AnalyticFreeBd}. 
   
    \subsection{Main results}
   In this paper we prove existence and regularity results for the one-phase problem \eqref{def:def-viscosity-solution}. 
In \Cref{sub:intro-existence}, we discuss the existence of viscosity solution to \eqref{def:def-viscosity-solution} via the Perron's method, while in \Cref{sub:intro-regularity}, we state our regularity results (flatness implies $C^{2,\alpha}$ and higher regularity of the free boundary) in the case when the solution is flat enough.

In what follows, we will say that a constant is universal, if it depends on one or more of the quantities $d$, $\beta$, $\lambda$, $\Lambda$, $\gamma_0$, $[\F]_{\text{Lip}(\R^d)}$, $[\F]_{C^{0,\beta}(B_1)}$, $\|f\|$, $\|g\|$, $\mathcal{A}$, $\underline u$, defined in \ref{hyp:h1}, \ref{hyp:h3}, \cref{def:a-supersolution}, \cref{def:strictminorant}.

\subsubsection{\textbf{Existence of viscosity solutions}}\label{sub:intro-existence} 
We prove the existence of viscosity solution to \eqref{def:def-viscosity-solution} via the Perron's method, taking the pointwise infimum of a suitable family $\mathcal A$ of supersolutions, which stay above a particular subsolution $\underline u$ (called strict minorant) and above a boundary datum. 
We use the following definitions.
    \begin{definition}[Admissible family of supersolutions $\mathcal{A}$]\label{def:a-supersolution}
		We say that a function $w:\overline B_1\to\R$ belongs to the class $\mathcal{A}$ of admissible supersolutions for the problem \eqref{def:def-viscosity-solution} if the following conditions hold:
		\begin{itemize}
			\item[i)] $w$ is continuous and non-negative in $\overline B_1$;\smallskip
			\item[ii)] $w$ satisfies $$\F(D^2w,\nabla w,x)\le f\quad\text{in }\Omega_{w}\cap B_1$$ in viscosity sense;\smallskip
			\item[iii)] for every $x_0\in \partial \Omega_w\cap B_1$, then:
           \begin{itemize}
            \item[a)] either there is $\widetilde\varphi\in C^1(B_1)$ such that $\widetilde\varphi_+$ touches $w$ from above at $x_0$ and 
			$$0\neq |\nabla \widetilde\varphi(x_0)|<g\left(x,\frac{\nabla\widetilde\varphi(x_0)}{|\nabla\widetilde\varphi(x_0)|}\right);$$
            \item[b)] or $$w=o(|x-x_0|)\quad\text{as }x\to x_0.$$
            \end{itemize}
		\end{itemize}
	\end{definition}

    \begin{definition}[Strict minorant]\label{def:strictminorant}
		We say that a function $\underline u:\overline B_1\to\R$ is a strict minorant for the problem \eqref{def:def-viscosity-solution} if the following conditions hold:
		\begin{itemize}
			\item[i)] $\underline u$ is continuous and non-negative in $\overline B_1$;\smallskip
			\item[ii)] $\underline u$ satisfies $$\F(D^2\underline u,\nabla \underline u,x)\ge f\quad\text{in } \Omega_{\underline u}\cap B_1$$ in viscosity sense;\smallskip
			\item[iii)] for every $x_0\in \partial \Omega_{\underline u}\cap B_1$, there is $\widetilde\varphi\in C^1(B_1)$ that touches $\underline u$ from below at $x_0$ and
			$$0\neq |\nabla \widetilde\varphi(x_0)|> g\left(x,\frac{\nabla\widetilde\varphi(x_0)}{|\nabla\widetilde\varphi(x_0)|}\right).$$ 
		\end{itemize}
	\end{definition}
  It is immediate to verify that the functions in the admissible class $\mathcal{A}$ are supersolution to \eqref{def:def-viscosity-solution}, while any strict minorant $\underline u$ is a subsolution to \eqref{def:def-viscosity-solution}, according with \cref{def:def-sol} (see \cref{rem:supersol}).

\smallskip
Our main existence theorem is the following.
	\begin{theorem}[Existence of viscosity solutions]\label{t:existence}
		Suppose that $\F\in\EE$, $f\in C^0( \overline B_1)$ and $g\in {C^0}(\overline B_1,\mathbb{S}^{d-1})$ satisfies the hypothesis \ref{hyp:h1}. Let $\phi:\partial B_1\to\R$ be a non-negative continuous function, $\mathcal{A}$ be the set of admissible supersolutions from \cref{def:a-supersolution} and $\underline u$ be a strict minorant from \cref{def:strictminorant}, with $\underline u\le \phi$ on $\partial B_1$ (e.g. $\underline u\equiv0$).
    Then the family \be\label{def:a-u-phi}\mathcal{A}_{\underline u,\phi}:=\{w\in \mathcal{A}: \ w\ge \underline  u\text{ in } \overline B_1, \ w\ge \phi\text{ on } \partial B_1\}\ee is non-empty and the function $u:\overline B_1\to\R$ defined as
  \begin{equation}\label{e:def-viscosity-solution-u}
  u(x):=\inf\{w(x):w\in \mathcal{A}_{\underline u,\phi}\}
  \end{equation} is a viscosity solution to \eqref{def:def-viscosity-solution}, according with \cref{def:def-sol}, with $u\in C^0(\overline B_1)$ and $u=\phi$ on $\partial B_1$.
  Moreover, in every compact $K\subset\subset B_1$, $u$ is Lipschitz continuous and non-degenerate, in the sense that 
        $$\|\nabla u\|_{L^\infty(K)}\le L\qquad\text{and}\qquad\sup_{B_r(x)}u\ge cr\quad\text{for every} \quad x\in \overline\Omega_u\cap K,\quad r\in(0,r_0),$$ for some universal constants $L>0$, $c>0$ and $r_0>0$, which also depend on $K$.
\end{theorem}
In \Cref{sub-ex} we will discuss the main steps for the proof of \cref{t:existence}. We also refer to \Cref{remarksss} for some remarks about the Perron's solutions constructed in \cref{t:existence}, as well as about the family of admissible supersolutions $\mathcal A$ and the strict minorant $\underline u$.

    \subsubsection{\textbf{Regularity of the free boundary}}\label{sub:intro-regularity}
    Our results about the regularity of the free boundary $\partial \Omega_u\cap B_1$ apply to viscosity solutions $u$ to \eqref{def:def-viscosity-solution}, even ones not obtained via the Perron's method. 
    As a first regularity result, via an improvement of flatness argument, we prove that if $u$ is flat enough, then $u$ and the free boundary $\partial\Omega_u\cap B_1$ are both locally of class $C^{2,\alpha}$. 
    
	\begin{theorem}[Flatness implies $C^{2,\alpha}$] \label{t:flatness-implies} There are universal constants $\eps_0>0$ and $\alpha_0\in(0,1)$ such that the following holds.
        Suppose that $\F\in\EE,$ $ f\in C^{0,\beta}(B_1),$ $ g\in C^{1,\beta}(B_1,\mathbb{S}^{d-1})$ satisfy the hypotheses \ref{hyp:h1}, \ref{hyp:h2}, \ref{hyp:h3} and let $\alpha\in(0,1)$ be such that 
        \bea\label{def:alpha}
        \alpha<\min\{\alpha_0,\beta\}.\eea 
        Let $u:B_1\to \R$ be a non-negative continuous viscosity solution to \eqref{def:def-viscosity-solution} such that $$\left(g(0,\nu)(x\cdot \nu)-\eps_0\right)_+\le u(x)\le \left(g(0,\nu)(x\cdot \nu)+\eps_0\right)_+\quad\text{for every}\quad x\in B_1,$$ for some unit vector $\nu\in\mathbb{S}^{d-1}$. Then $u\in C^{2,\alpha}(\overline\Omega_u\cap B_{1/2})$ and the free boundary $\partial \Omega_u\cap B_1$ is a $(d-1)$-dimensional manifold of class $C^{2,\alpha}$ in $B_{1/2}$.
	\end{theorem}    
    We refer to \Cref{sub-reg} for a discussion about the main steps for the proof of \cref{t:flatness-implies}, which is based on iteration of the quadratic improvement of flatness in \cref{t:improvement of flatness}.
    
    We point out that, even if $\F$, $f$ and $g$ are less regular, the argument from \cref{t:flatness-implies} can be used to show that the flat free boundaries are $C^{1,\alpha}$. Precisely, if $\mathcal F\in\EE$, $f\in C^0(B_1)$ and $g\in C^{0,\beta}(B_1\times \mathbb{S}^{d-1})$ satisfies the hypothesis \ref{hyp:h1},
then the free boundary $\partial \Omega_u\cap B_1$ is $C^{1,\alpha}$-regular in $B_{1/2}$ and $u\in C^{1,\alpha}(\overline\Omega_u\cap B_{1/2})$. This result was proved in \cite{desilvaferrarisalsafullynonlinear} for the two-phase problem.

 \smallskip
    Once the solution $u$ and its free boundary $\partial\Omega_u $ are $C^{2,\alpha}$-regular, we can use the hodograph transform from  \cite[Theorem 2]{KinderlehrerNirenberg1977:AnalyticFreeBd} to reduce the free boundary problem \eqref{def:def-viscosity-solution} to an elliptic nonlinear equation with oblique boundary condition on a fixed boundary. This gives the following higher regularity result.

    \begin{corollary}[Higher regularity]\label{corollary:cinfty}
        Suppose that $\F\in C^{k,\beta}(\mathcal{S}^{d\times d}\times \R^d\times B_1)\cap \EE,$ $f\in C^{k,\beta}(B_1),$ $ g\in C^{k+1,\beta}(B_1,\mathbb{S}^{d-1})$ satisfy the hypotheses \ref{hyp:h1}, \ref{hyp:h2}, for some $k\in\N_{\ge1}$ and $\beta\in(0,1)$. 
        Let $u:B_1\to \R$ be a non-negative continuous viscosity solution to \eqref{def:def-viscosity-solution} such that $$\left(g(0,\nu)(x\cdot \nu)-\eps_0\right)_+\le u(x)\le \left(g(0,\nu)(x\cdot \nu)+\eps_0\right)_+\quad\text{for every}\quad x\in B_1,$$ for some unit vector $\nu\in\mathbb{S}^{d-1}$ and $\eps_0>0$ as in \cref{t:flatness-implies}. Then the free boundary $\partial \Omega_u\cap B_1$ is a $(d-1)$-dimensional manifold of class $C^{k+2,\beta}$ in $B_{1/2}$.
		Moreover, if $\F,$ $f$, $g$ are $C^{\infty}$ (resp. analytic), then the free boundary $\partial \Omega_u\cap B_1$ is $C^{\infty}$ (resp. analytic) in $B_{1/2}$.
	\end{corollary}
    
     We notice that \cite[Theorem 2]{KinderlehrerNirenberg1977:AnalyticFreeBd} requires the $C^2$ regularity of $u$ as this is the minimum hypothesis one needs in order to deduce higher regularity for solutions of nonlinear elliptic equations (see \cite[Theorem 11.1]{adn} and \cite[Section 6.7]{morrey-multiple-integrals}). In the case of divergence form operators, in \cite{KriventsovLin} it was shown that the $C^{1,\alpha}$ regularity is sufficient to run the argument from \cite{KinderlehrerNirenberg1977:AnalyticFreeBd} and to obtain (in the case of smooth data) the $C^\infty$ regularity. To our knowledge, for fully nonlinear operators, there are no available results that allow to bootstrap directly from $C^{1,\alpha}$ to $C^\infty$. This is the reason for which we first prove the $C^{2,\alpha}$ regularity via an improvement of flatness argument.

\subsection{Sketch of the proof and some remarks.}
Here below we discuss the main steps for the proof of our existence and regularity results \cref{t:existence} and \cref{t:flatness-implies}, with some remarks about the Perron's solution constructed in \cref{t:existence}.

\subsubsection{Existence of viscosity solutions}\label{sub-ex} We prove \cref{t:existence} via the Perron's method; the strategy is similar to the one from \cite{Caffarelli1988:HarnackApproachFreeBd3Existence,wangIII,DESILVA-perron,salsatuloneverzini-existenceviscositysol} developed for two-phase problems.
We divide the proof in five steps.\smallskip 

\textit{Step 1.} In the first step we prove that $u$ is Lipschitz continuous and satisfies \eqref{def:def-viscosity-solution} in $\Omega_u$ in \cref{prop:lip-and-sol}. More precisely, first in \cref{lemma:ex1} we show that, given an admissible supersolution $w\in\mathcal{A}_{\underline u,\phi}$, we can replace $w$ with another supersolution $\overline w\in\mathcal{A}_{\underline u,\phi}$ which satisfies the interior condition
       $$\F(D^2\overline w,\nabla\overline w,x)=f\quad\text{in }\Omega_{\overline w}\cap B_1,$$ solving an obstacle-type problem for a fully nonlinear operator. We notice that the dichotomy in $iii)$ of \cref{def:a-supersolution} is crucial for the existence of this replacement (see \Cref{remarksss}).
Second, in \cref{lemma:ex2}, we use Harnack inequality and comparison arguments with the subsolution from \cref{lemma:barrier1} as a barrier, to show that the replaced supersolutions $\overline w$ are locally equi-Lipschitz, and so $u$ is locally Lipschitz. As a corollary, we obtain that $u$ is locally obtained as uniform limit of admissible supersolutions (see \cref{corollary:ex0}).\smallskip

\textit{Step 2.} In the second step we prove that $u$ is a supersolution to \eqref{def:def-viscosity-solution} in \cref{prop:supersol}. 
This essentially follows from the fact that $u$ is locally obtained as limit of a sequence of supersolutions $\{w_k\}\subset\mathcal{A}_{\underline u,\phi}$.\smallskip

\textit{Step 3.} In the third step we show that $u$ is non-degenerate in \cref{prop:non-and-sol}. As a fundamental ingredient, we use a strong maximum principle \cref{lemma:maximum}, which follows by $iii)$ of \cref{def:strictminorant}, to ensure that the free boundary of $u$ and the free boundary of $\underline u$ cannot intersect. Arguing by contradiction, we show that if $u$ is degenerate in $0\in\Omega_u$, then we can construct an admissible supersolution as the minimum between some $w\in \mathcal A_{\underline u,\phi}$ close to $u$ and the supersolution from \cref{lemma:barrier2} which vanishes in $0$. \smallskip

\textit{Step 4.} In the fourth step we prove that $u$ is a subsolution to \eqref{def:def-viscosity-solution} in \cref{prop:subsol}. The idea is that if a test function $\varphi_+$ with small gradient touches $u$ from above at $0\in\partial\Omega_u$, then at a small scale, $u$ is touched from above by a function $h$, which is a translation of an half-plane solution.
This can be used to construct an admissible supersolution, obtained as the minimum between some $w\in \mathcal A_{\underline u,\phi}$ close to $u$ and a small perturbation of $h$ which vanishes in a neighborhood of $0$. This is a contradiction with $0\in\partial\Omega_u$.\smallskip

\textit{Step 5.} The fifth step is dedicated to the proof of the continuity of $u$ up to the boundary in \cref{prop:attacca} and the fact that $\mathcal{A}_{\underline u,\phi}$ is non-empty (and so $u$ is well-defined) in \cref{prop:construction}. The continuity of $u$ up to the boundary follows from classical exterior ball arguments, while $\mathcal{A}_{\underline u,\phi}\not=\emptyset$ is obtained by the construction of a particular admissible supersolution $w\in\mathcal{A}_{\underline u,\phi}$.

\subsubsection{Regularity of the free boundary}\label{sub-reg} We prove \cref{t:flatness-implies} by an improvement of flatness argument by approximating the solutions with second degree polynomials. The argument is similar to the one from \cite[Theorem 1.1]{dfs-higher-order-two-phase} (see also \cite{DESILVAc2alpha,DesilvaSavin:globalsolution}), which is also obtained via a quadratic improvement of flatness for the Laplace operator, even though the second order approximations are different.

        For any unit vector $\nu\in\mathbb{S}^{d-1}$, we consider quadratic polynomials which are solutions to the following fully nonlinear elliptic equation with oblique boundary condition
        \be\label{def:sol-pol}\begin{cases}
				\F(D^2p,g(0,\nu)\nu,0)=f(0)\quad&\text{in } B_1,\\
				\nabla p(x)\cdot \tau(\nu,g)=\nabla_x g(0,\nu)\cdot x\quad&\text{in } B_1\cap \{x\cdot \nu=0\},
			\end{cases}\ee 
            where 
            \be\label{def:tau}\tau(\nu,g):=\nu+\frac{(\nabla_\theta g(0,\nu)\cdot \nu)\nu}{g(0,\nu)}
			-\frac{\nabla_\theta g(0,\nu)}{g(0,\nu)}.\ee
        For simplicity, we use the notation \be\label{def:polynomial}\begin{aligned}\mathcal{P}(\nu,\F,f,g):=\{p(x)=Mx\cdot x :\ M\in\R^{d\times d},\ \text{$p$ is a solution to \eqref{def:sol-pol}}\}.\end{aligned}\ee  
        In the quadratic improvement of flatness \cref{t:improvement of flatness}, we show that, under some structural hypotheses on $\F$, $f$ and $g$, if there is a quadratic polynomial 
        $p\in\mathcal{P}(e_d,\F,f,g)$ such that $\|p\|_{L^\infty(B_1)}\le r$ and
					$$\left(g(0,e_d)x_d+p(x)-r^{1+\alpha}\right)_+\le u(x)\le \left(g(0,e_d)x_d+p(x)+r^{1+\alpha}\right)_+\quad\text{for every}\quad x\in B_1,$$  
        then there are a unit vector $\nu'\in\mathbb{S}^{d-1}$ and a quadratic polynomial $p'\in\mathcal{P}(\nu',\F, f, g)$ such that
				$$\left(g(0,\nu')(x\cdot\nu')+\rho p'(x)-(r\rho)^{1+\alpha}\right)_+\le u_{\rho}(x)\le \left(g(0,\nu')(x\cdot\nu')+\rho p'(x)+(r\rho)^{1+\alpha}\right)_+$$ 
                for every $x\in B_1$, where $u_\rho(x):=\frac{u(\rho x)}{\rho}$, for some $\rho\in(0,1)$.
     
      Following the strategy from \cite{DeSilva:FreeBdRegularityOnePhase}, we argue by contradiction and we prove a partial Harnack inequality in \cref{prop:partial-harnack}, which assures that the linearized sequence converges locally uniformly. In \cref{lemma:linearized-problem}, we show that the limit function $\widetilde u$ satisfies the following fully nonlinear problem with oblique boundary condition
        $$\begin{cases}
				\widetilde \F(D^2\widetilde u)=0 \quad&\text{in }B_{1/2}^+,\\
				\nabla\widetilde u\cdot \tau=0\quad&\text{on }B_{1/2}'.
			\end{cases}$$ 
        By the convexity/concavity assumption \ref{hyp:h2} on $\mathcal F$, the limit function $\widetilde u$ is of class $C^{2,\alpha_0}$, for some $\alpha_0\in(0,1)$, by \cref{prop:c2alphaestimate}; in fact, the constant $\alpha_0$ from \cref{t:flatness-implies} is precisely the constant from \cref{prop:c2alphaestimate}. Finally, in \Cref{section-contradiction}, we adjust the vector $\nu$ and the polynomial $p$ using $\nabla \widetilde u(0)$ and $D^2\widetilde u(0)$ and modifying them appropriately so that the new vector $\nu'$ and the new quadratic polynomial $p'$ belong to $\mathbb{S}^{d-1}$ and $\mathcal{P}(\nu',\F,f,g)$ respectively.     
        
        \cref{t:flatness-implies} is a consequence of the quadratic improvement of flatness in \cref{t:improvement of flatness} and an iteration procedure, which starts when $u$ is sufficiently flat. 
        In order to prove the $C^{2,\alpha}$ regularity of the free boundary, we apply this procedure at every point $x_0\in\partial\Omega_u\cap B_{1/2}$ to the rescalings 
        $$u_{x_0,\rho}(x):=\frac{u(x_0+\rho x)}{\rho},$$
       which are solutions to \eqref{def:def-viscosity-solution} with operator $\F_{x_0,\rho}$, right-hand side $f_{x_0,\rho}$, and free boundary condition $g_{x_0,\rho}$, where 
       \be\label{def:rescaling-operator}\begin{aligned}
       \F_{x_0,\rho}(M,\xi,x)&:=\rho\F\left(\rho^{-1}M,\xi,x_0+\rho x\right),\\
       f_{x_0,\rho}(x)&:=\rho f(x_0+\rho x),\\ 
       g_{x_0,\rho}(x,\nu)&:=g(x_0+\rho x,\nu).
	\end{aligned}\ee 
				Thanks to the assumptions on $\F$, $f$, $g$, the rescalings $\F_{x_0,\rho}$, $f_{x_0,\rho}$, $g_{x_0,\rho}$ satisfy the hypotheses of \cref{t:improvement of flatness} (see \cref{lemma:rate}).
				Moreover, since $p'\in\mathcal{P}(\nu',\F,f,g)$, the new polynomial $\rho p'$ satisfies $$\rho p'\in\mathcal{P}(\nu',\F_{x_0,\rho},f_{x_0,\rho},g_{x_0,\rho})\quad\text{and}\quad \|\rho p'\|_{L^\infty(B_1)}\le \rho r,$$
                which allows to iterate the improvement of flatness.
                This gives the second order Taylor expansion of $u$, with uniform rate of convergence, at any $x_0\in\partial\Omega_u\cap B_{1/2}$, and so \cref{t:flatness-implies} is proved.
   
\subsubsection{Remarks about the Perron's solutions}\label{remarksss}
Here below we make some remarks on the solutions obtained via the construction in \cref{t:existence}, and on the definitions of the family $\mathcal A$ and the strict minorant $\underline u$.\smallskip

\textit{(1) Existence for any boundary datum.}
We notice that, for homogeneous and concave $\mathcal F=\mathcal F(D^2u)$, viscosity solutions to the one-phase problem \eqref{def:def-viscosity-solution} can be obtained also via \cite{salsatuloneverzini-existenceviscositysol} by taking non-negative continuous boundary datum $\phi:\partial B_1\to\R$ and, as strict minorant (in the sense of \cite{salsatuloneverzini-existenceviscositysol}), a function $\underline u$ such that $\underline u=\phi$ on $\partial B_1$. The existence of such strict minorant is not a priori guaranteed for general boundary data $\phi$. \cref{t:existence} allows for general (even not homogeneous and concave) $\mathcal F=\mathcal F(D^2u,\nabla u,x)$ and provides existence of viscosity solutions for any (non-negative and continuous) boundary datum $\phi$. Indeed, since in \cref{t:existence} we only require that $\underline u\le \phi$ on $\partial B_1$, we can always take as strict minorant the constant function $\underline u\equiv 0$. \smallskip 

\textit{(2) The Perron's solutions are a special class of viscosity solutions.} 
We point out that all the solutions obtained this way belong to a special class of viscosity solutions and that not all viscosity solutions to \eqref{def:def-viscosity-solution} fall in this class.
The properties of the solutions $u$ obtained by \eqref{e:def-viscosity-solution-u} are a consequence of the choice of the particular family of supersolutions $\mathcal A$ and of the strict minorant $\underline u$ as a special subsolution. Moreover, the solution $u$ defined in \eqref{e:def-viscosity-solution-u} depends on the choice of the strict minorant $\underline u$ and so \cref{t:existence} provides different viscosity solutions for different choices of the strict minorant.

The fact that the functions in $\mathcal A$ satisfy $iii)$ of \cref{def:a-supersolution} allows to exclude solutions with collapsed free boundaries, as for instance $u(x)=c|x_d|$. In fact, for any $0<c\le 1$, the function $u(x)=c|x_d|$ is a viscosity solution to \eqref{def:def-viscosity-solution} with $\F=\F(D^2 u)$, $f\equiv0$ and $g\equiv 1$ (since $u$ cannot be touched from above by a test function at points on the zero set $\{x_d=0\}$), but it cannot be obtained by \eqref{e:def-viscosity-solution-u}; indeed, by $iii)$ of \cref{def:a-supersolution}, $w>0$ in $B_1$ for every $w\in\mathcal{A}$ and thus $u$ is a uniform limit (see \cref{corollary:ex0}) of functions that satisfy $\F(D^2 w)\ge 0$ in $B_1$; then $\F(D^2 u)\ge 0$ in $B_1$ and this is a contradiction with the maximum principle for the operator $\F$.

For what concerns the choice of the strict minorant, the fact that $\underline u$ satisfies $iii)$ of \cref{def:strictminorant} is crucial in the proof of the non-degeneracy of the viscosity solution $u$ defined in \eqref{e:def-viscosity-solution-u} (see \cref{prop:non-and-sol}), a property that also restricts the class of viscosity solutions obtained in \cref{t:existence}. For instance, the free boundary problem \eqref{def:def-viscosity-solution} admits also degenerate viscosity solutions, an example being the function $u(x)=\frac12x_d^2$, which is a viscosity solution to \eqref{def:def-viscosity-solution} with $\F=\Delta$, $f\equiv1$ and $g\equiv 1$. \cref{t:existence} guarantees that such degenerate solutions cannot be obtained by \eqref{e:def-viscosity-solution-u}.\smallskip 

\textit{(3) The dichotomy in \cref{def:a-supersolution}.}
The dichotomy in $iii)$ of \cref{def:a-supersolution} has a special role in the proof of \cref{t:existence} as it allows to deal with the presence of the right-hand side $f$ in \eqref{def:def-viscosity-solution}. Indeed, given an admissible supersolution $w\in\mathcal{A}_{\underline u,\phi}$, we replace $w$ with another supersolution $\overline w\in\mathcal{A}_{\underline u,\phi}$ which satisfies the interior condition
       $$\F(D^2\overline w,\nabla\overline w,x)=f\quad\text{in }\Omega_{\overline w}\cap B_1.$$ To ensure the non-negativity of the replacement $\overline w$, we need to consider an obstacle-type problem for a fully nonlinear operator (see \cref{lemma:ex1}). The replacement $\overline w$ satisfies $\Omega_{\overline w}\subset \Omega_w$ and this inclusion might be strict, so there might be points on the free boundary $\partial\Omega_{\overline w}\cap B_1$, which are not lying on $\partial\Omega_{w}\cap B_1$. In particular, at these points there might not be a function $\widetilde\varphi$ satisfying the condition $a)$ of \cref{def:a-supersolution}. On the other hand, by the $C^{1,\gamma}_{loc}$ regularity of the obstacle problem, at these points we have $\overline w=o(|x-x_0|)$, so $b)$ of \cref{def:a-supersolution} is fulfilled.

\smallskip
\textit{(4) The Perron's solution can be strictly positive.} We observe that the Perron's solution $u$ defined in \eqref{e:def-viscosity-solution-u} might also be strictly positive in the whole $B_1$; this can happen for instance if the boundary datum $\phi$ is large, or if the right-hand side $f$ is negative and large. Still, if there exists $w\in\mathcal{A}_{\underline u,\phi}$ such that the zero set $\{w=0\}\cap B_1$ is not empty, then also the free boundary $\partial\Omega_u\cap B_1$ is non-empty. \smallskip

    \textit{(5) Lipschitz regularity and non-degeneracy.}
     Finally, we notice that every viscosity solution $u$ to \eqref{def:def-viscosity-solution}, even if not obtained via the Perron's method, is locally Lipschitz continuous (see \cref{rem:every-sol-is-lip}), while for the non-degeneracy this is not true; indeed, one example of a degenerate viscosity solution is the function $u(x)=\frac12x_d^2$ mentioned above.	

	\subsection*{Acknowledgements} The authors are supported by the European Research Council (ERC), through the European Union's Horizon 2020 project ERC VAREG - \it Variational approach to the regularity of the free boundaries \rm (grant agreement No. 853404); they also acknowledge the MIUR Excellence Department Project awarded to the Department of Mathematics, University of Pisa, CUP I57G22000700001. B.V. acknowledges also support from the projects PRA 2022 14 GeoDom (PRA 2022 - Università di Pisa) and MUR-PRIN ``NO3'' (No. 2022R537CS).
	\section{Preliminaries}\label{section2}
	\subsection{Notations} 
	We denote by $B_r(x_0)$ the ball centered in $x_0\in\R^d$ with radius $r>0$ in $\R^d$. We also set $B_r:=B_r(0)$. Given $E\subset\R^d$, we denote by $$E^+:=E\cap \{x_d>0\}\quad\text{and}\quad E':=E\cap \{x_d=0\}.$$ We denote a point $x\in\R^d$ as $x=(x',x_d)\in\R^{d-1}\times\R^d$. We say that a constant is universal, if it depends on one or more of the quantities $d$, $\beta$, $\lambda$, $\Lambda$, $\gamma_0$, $[\F]_{\text{Lip}(\R^d)}$, $[\F]_{C^{0,\beta}(B_1)}$, $\|f\|$, $\|g\|$, $\mathcal{A}$, $\underline u$, defined in \ref{hyp:h1}, \ref{hyp:h3}, \cref{def:a-supersolution}, \cref{def:strictminorant}.
	\subsection{Viscosity settings}\label{section2.1} We work with viscosity solutions to the problem \eqref{def:def-viscosity-solution}, so we recall the main definitions for viscosity settings.
	Let $E\subset \R^d$ an open set and $v,w\in C^0(E)$, then:
	\begin{itemize}
		\item we say that $v$ touches $w$ from below at $x_0\in E$ if $$v(x_0)=w(x_0)\quad\text{and} \quad v\le w \quad\text{in a neighborhood of $x_0$}$$ (strictly if the inequality is strict);\smallskip
		\item we say that $v$ touches $w$ from above at $x_0\in E$ if $$v(x_0)=w(x_0)\quad\text{and} \quad v\ge w \quad\text{in a neighborhood of $x_0$} $$ (strictly if the inequality is strict).
	\end{itemize}
	Now we recall the definition of viscosity solutions of an elliptic equation in an open set.
	
	\begin{definition}[Viscosity solutions of elliptic equations] Let $E\subset \R^d$ an open set, $\F\in\EE$ and $f\in C^0(E)$. We say that $v$ is a viscosity subsolution (resp. supersolution) to the elliptic equation $\F(D^2v,\nabla v,x)=f$, and we write $$\F(D^2v,\nabla v,x)\ge f\quad\text{in } E\qquad\text{(resp $\le$)},$$
		if, for every $x_0\in E$ and for every $\varphi\in C^2(E)$ such that $\varphi$ touches $u$ from above (resp. below), then it holds $$\F(D^2\varphi(x_0),\nabla \varphi(x_0),x)\ge f(x_0)\qquad\text{(resp $\le$)}.$$ 
		We say that $v$ is a viscosity solution to $$\F(D^2v,\nabla v,x)= f\quad\text{in } E,$$ if it is both a viscosity subsolution and viscosity supersolution.
	\end{definition}

		\subsection{Remarks on the admissible family $\mathcal{A}$ and the strict minorant $\underline u$} 
        In the following lemma, we prove that the functions in the admissible family $\mathcal{A}$ from \cref{def:a-supersolution} are supersolutions to \eqref{def:def-viscosity-solution}, as well as, the strict minorant $\underline u$ from \cref{def:strictminorant} is subsolution to \eqref{def:def-viscosity-solution}. 
        \begin{lemma}\label{rem:supersol} 
        Suppose that $\F\in\EE$, $f\in C^0( B_1)$, $g\in {C^0}( B_1,\mathbb{S}^{d-1})$. Let $\mathcal A$ be the family from \cref{def:a-supersolution} and $\underline u$ be the strict minorant from \cref{def:strictminorant}. 
       \begin{itemize}
           \item If $w\in\mathcal{A}$, then $w$ is a supersolution to \eqref{def:def-viscosity-solution}, according to \cref{def:def-sol}.\smallskip
           \item $\underline u$ is a subsolution to \eqref{def:def-viscosity-solution}, according to \cref{def:def-sol}.
       \end{itemize}
        \end{lemma}
        \begin{proof}
        We only prove the first point, the proof of the second point being analogous.
        Given $w\in\mathcal{A}$, we only need to prove that $w$ satisfies the supersolution condition on the free boundary, according to \cref{def:def-sol}. Let $\varphi\in C^1(B_1)$ which touches $w$ from below at $x_0\in \partial\Omega_{w}\cap B_1$, with $|\nabla\varphi(x_0)|\not=0$, then there is a function $ \widetilde\varphi\in C^1(B_1)$, such that $\widetilde\varphi_+$ touches $w$ from above at $x_0$ and $|\nabla\widetilde \varphi(x_0)|\not=0$, otherwise $$\nabla\varphi(x_0)\cdot (x-x_0)+o(|x-x_0|)=\varphi(x)\le  w(x)=o(|x-x_0|)\quad\text{as }x\to x_0,$$ which is a contradiction.
    Therefore $$\nabla\varphi(x_0)\cdot (x-x_0)+o(|x-x_0|)=\varphi\le w\le \widetilde\varphi_+=\left(\nabla\widetilde\varphi(x_0)\cdot (x-x_0)+o(|x-x_0|)\right)_+\quad\text{as }x\to x_0.$$ Choosing $x=x_0+|x-x_0|\nu$ in the last inequality, where $\nu\in\R^d$ is a unit vector such that $\nabla\widetilde\varphi(x_0)\cdot\nu>0$, we get $$\nabla\varphi(x_0)=\nabla\widetilde\varphi(x_0).$$
    Then
    $$|\nabla \varphi(x_0)|=|\nabla \widetilde \varphi(x_0)|< g\left(x_0,\frac{\nabla \widetilde\varphi(x_0)}{|\nabla\widetilde \varphi(x_0)|}\right)=g\left(x_0,\frac{\nabla \varphi(x_0)}{|\nabla \varphi(x_0)|}\right),$$ i.e.~$w$ is a supersolution.
    \end{proof}
    In the next lemma, we prove that the free boundaries of the strict minorant $\underline u$ from \cref{def:strictminorant} and of a supersolution to \eqref{def:def-viscosity-solution} cannot intersect. This lemma will be crucial in the proof of the non-degeneracy property of the function $u$ defined in \eqref{e:def-viscosity-solution-u} (see \cref{prop:non-and-sol}).
    \begin{lemma}[A strong maximum principle]\label{lemma:maximum}
    Suppose that $\F\in\EE$, $f\in C^0(  B_1)$, $g\in {C^0}( B_1,\mathbb{S}^{d-1})$. Let $\underline u$ be the strict minorant from \cref{def:strictminorant} and $v$ be a supersolution in the sense of \cref{def:def-sol}. If $v\ge \underline u$ in $B_1$, then either $v>\underline u$ on $\overline\Omega_{\underline u}\cap B_1$ or $v\equiv \underline u>0$ in $B_1$.
    \end{lemma}
    \begin{proof}  
    Suppose that $v$ touches $\underline u$ from above at a point $x_0\in B_1\cap\overline\Omega_{\underline u}$. If $x_0\in \partial\Omega_{\underline u}$, then it also holds $x_0\in\partial\Omega_v$. By definition of $\underline u$, there should be a function $\widetilde\varphi$ as in \cref{def:strictminorant}, which is a contradiction with the fact that $v$ is a supersolution. If $x_0\in \Omega_{\underline u}$, then both $v$ and $\underline u$ are positive in a neighborhood of $x_0$, so from the maximum principle for the operator $\mathcal F$, we have that $v\equiv \underline u$ on $\overline\Omega_{\underline u}\cap B_1$. In particular, there is a boundary point $y_0\in\partial\Omega_v\cap\partial\Omega_{\underline u}\cap\overline B_1$. If $y_0\in B_1$, then the conclusion follows from the previous case, otherwise $v\equiv \underline u >0$ in $B_1$. 
    \end{proof}

        \subsection{Interior Harnack 
		inequality} 
        We recall the interior Harnack inequality for viscosity solutions of a fully nonlinear operator with right-hand side.
		\begin{proposition}[Interior Harnack inequality]\label{prop:interiorharnack} Let $\F\in\EE$ and let $v:B_{2\rho}(x_0)\to\R$ be a non-negative continuous viscosity solution to $$
			\F(D^2v,\nabla v,x)=f\quad\text{in } B_{2\rho}(x_0),$$ for some $x_0\in B_1$ and $\rho>0$. Then $$\sup_{B_\rho(x_0)}v\le C\left(\inf_{B_\rho(x_0)}v+\|f\|_{L^\infty(B_{2\rho}(x_0))}\right)$$
			for some universal constant $C>0$.
		\end{proposition}
		\begin{proof}
			See \cite[Theorem 4.3]{caffarellicabre} or \cite[Theorem 3.5]{AmendolaRossiVitolo-Harnack}.
		\end{proof}

		\subsection{Construction of barriers}
		In this subsection we construct two barriers that will be useful in the proof of \cref{t:existence}.
		\begin{lemma}[A radial subsolution]\label{lemma:barrier1}
			Given $\rho>0$, there is a non-negative continuous function $\psi:\overline B_\rho\to\R$ such that the following holds. Suppose that $ \F\in\EE$ and let $\F_r:=\F_{x_0,r}$ be the rescaling as in \eqref{def:rescaling-operator}. Then, for every $\delta>0$ there exists a constant $\sigma=\sigma(\delta)>0$ such that the function $\zeta:=\delta\psi$ satisfies
			$$\begin{cases}
				\F_{r}( D^2\zeta, \nabla\zeta,x)> \sigma\quad&\text{in } B_\rho\setminus\overline B_{\rho/2},\\
				|\nabla\zeta|\not=0\quad&\text{in } B_{\rho}\setminus\overline B_{\rho/2},\\
				\zeta=0\quad&\text{on }\partial B_{\rho},\\
				\zeta= \delta \quad&\text{in } B_{\rho/2},
			\end{cases}$$ for every $r\in(0,1]$. 
		\end{lemma}
		\begin{proof}
			For simplicity, we set $\rho=1$. Let $\delta>0$ and $$\psi(x):=\begin{cases}
			    \widetilde C(e^{-\gamma|x|^2}-e^{-\gamma})\quad &\text{if } x\in B_1\setminus \overline B_{1/2},\\ 1\quad &\text{if } x\in B_{1/2}
			\end{cases}$$ where $\widetilde C>0$ is such that $\psi$ is continuous and $\gamma>0$ to be chosen later. Then $$\nabla \psi=-2\widetilde C\gamma xe^{-\gamma|x|^2}\quad\text{and}\quad D^2\psi=2\widetilde C\gamma (2\gamma x\otimes x-I).$$ In particular, in an appropriate system of coordinates $$D^2\psi=2\widetilde C\gamma\text{diag}(2\gamma,-1,\ldots-1).$$
			By \eqref{eq:pucci}, for every $r\in(0,1)$ and $x\in B_1\setminus \overline B_{1/2}$, we have \be\label{eq:richiama}\begin{aligned} \F_{r}(D^2\zeta,\nabla \zeta,x)&\ge\mathcal{M}^-(D^2\zeta)-r[\F]_{\text{Lip}(\R^d)}|\nabla \zeta|
			\\
			&=2\delta\widetilde C\gamma e^{-\gamma|x|^2}(2\gamma\lambda-\Lambda(d-1))-2\delta\widetilde C\gamma [\F]_{\text{Lip}(\R^d)}|x|e^{-\gamma|x|^2}.\\&\ge2\delta\widetilde C\gamma e^{-\gamma|x|^2}(2\gamma\lambda-\Lambda(d-1)-[\F]_{\text{Lip}(\R^d)})> \sigma,\end{aligned}\ee if we choose a universal $\gamma>0$ large enough and $\sigma=\sigma(\delta)>0$, as desired.
		\end{proof}
        
		\begin{lemma}[A radial supersolution]\label{lemma:barrier2}
			Given $\rho>0$, there is a non-negative continuous function $\widetilde\psi:\overline B_1\to\R$ such that the following holds. Suppose that $ \F\in\EE$ and let $\F_r:=\F_{x_0,r}$ be the rescaling as in \eqref{def:rescaling-operator}. Then, for every $\delta>0$ there exists a constant $\sigma=\sigma(\delta)>0$ such that the function $\widetilde \zeta:=\delta\widetilde \psi$ satisfies
			$$\begin{cases}
				\F_{r}( D^2\widetilde \zeta, \nabla\widetilde \zeta,x)<-\sigma\quad&\text{in } B_\rho\setminus\overline B_{\rho/2},\\
				|\nabla\widetilde \zeta|\not=0\quad&\text{in } B_{\rho}\setminus\overline B_{\rho/2},\\
				\widetilde \zeta=\delta\quad&\text{on }\partial B_{\rho},\\
				\widetilde \zeta= 0 \quad&\text{in } B_{\rho/2},
			\end{cases}$$ for every $r\in(0,1]$. 
		\end{lemma}
		\begin{proof}
			The proof follows choosing $\widetilde \psi:=1-\psi$ and arguing as in \cref{lemma:barrier1}.
		\end{proof}
        \subsection{\texorpdfstring{$C^{2,\alpha}$}{C2a} regularity for the linearized problem}
		The linearized problem in the proof of the quadratic improvement of flatness in \cref{t:improvement of flatness} satisfies a fully nonlinear elliptic equation with oblique boundary condition
        (see \cref{lemma:linearized-problem}). We need the following $C^{2,\alpha}$ estimates for viscosity solutions to this problem.
		\begin{proposition}[$C^{2,\alpha}$ estimate for linearized problem]\label{prop:c2alphaestimate}
        There are universal constants $C_0>0$ and $\alpha_0\in(0,1)$ such that the following holds.
			Let $\widetilde u:B_{1/2}^+\cup B_{1/2}'\to \R$ be a non-negative continuous viscosity solution to the following problem \be\label{eq:lin}\begin{cases}
				\widetilde \F(D^2\widetilde u)=0 \quad&\text{in }B_{1/2}^+,\\
				\nabla \widetilde u\cdot \tau=0\quad&\text{on }B_{1/2}',
			\end{cases}\ee
            where $\widetilde \F\in\EE$ satisfies \ref{hyp:h2} and $\tau\in\R^d$ satisfies $\tau\cdot e_d= 1$ and $|\tau|\le C$, where $C>0$ is a universal constant. If $\|u\|_{L^\infty}\le 1,$ then $$\left\|\widetilde u(x)-\widetilde u(0)-\nabla \widetilde u(0)\cdot x-\frac12 D^2\widetilde u(0)x\cdot x\right\|_{L^\infty(B_{\rho})}\le C_0\rho^{2+\alpha_0}$$ for every $\rho\in(0,1/4).$ 
		\end{proposition}
		\begin{proof}
			See \cite[Theorem 1.3]{obliqueregularity}.
		\end{proof}
        For the sake of completeness, we recall the definition of a viscosity solution to \eqref{eq:lin}. 
        We say that $\widetilde u:B_{1/2}^+\cup B_{1/2}'\to\R$ is a viscosity solution to \eqref{eq:lin} if $$\widetilde \F(D^2\widetilde u)=0\quad\text{in }B_{1/2}^+$$ in viscosity sense and for every $\varphi\in C^1(B_{1/2})$ such that $v$ touches $\varphi$ from below (resp. above) at $x_0\in B_{1/2}'$, we have that $$\nabla\varphi(x_0)\cdot \tau\ge 0\qquad\text{(resp. $\le$)}.$$
        We point out that the validity of being a viscosity solution to \eqref{eq:lin} can be verified using only polynomials as test functions.
        \section{Existence of solutions}\label{section3}
		In this section we prove \cref{t:existence}, namely, the existence of viscosity solutions to the one-phase problem \eqref{def:def-viscosity-solution}, according with \cref{def:def-sol}.
       For the two-phase problem the existence theory was developed in \cite{Caffarelli1988:HarnackApproachFreeBd3Existence,wangIII,DESILVA-perron,salsatuloneverzini-existenceviscositysol} (see also \cite{CaffarelliSalsa:GeomApproachToFreeBoundary}).
        We divide the proof of \cref{t:existence} in five steps.
        \subsection{Step 1. $u$ is Lipschitz and solves \eqref{def:def-viscosity-solution} in $\Omega_u$}\label{1step}
		In this subsection we prove the Lipschitz regularity of the function $u$ defined in \cref{t:existence} and we show that $u$ satisfies the interior condition in viscosity sense. In particular, we prove the following proposition.
		\begin{proposition}\label{prop:lip-and-sol} 
        Suppose that $\F\in\EE$, $f\in C^0( \overline B_1)$ and $g\in {C^0}(\overline B_1,\mathbb{S}^{d-1})$ satisfies the hypothesis \ref{hyp:h1}.
        Let $u:\overline B_1\to\R$ be the function from \cref{t:existence}, defined in \eqref{e:def-viscosity-solution-u}, then $u$ is locally Lipschitz continuous in $B_1$, namely, for every $\rho\in(0,1)$, we have $$\|\nabla u\|_{L^\infty(B_{\rho})}\le L.$$ for some universal constant $L>0$, which also depends on $\rho$. Moreover, $u$ satisfies the interior condition $$\F(D^2u,\nabla u,x)=f\quad\text{in }\Omega_u\cap B_1$$ in viscosity sense.
		\end{proposition}
	 
       In order to prove \cref{prop:lip-and-sol}, we use the following lemmas.
	 \begin{lemma}\label{rem:minimum} Suppose that $\F\in\EE$, $f\in C^0( \overline B_1)$ and $g\in {C^0}(\overline B_1,\mathbb{S}^{d-1})$ satisfies the hypothesis \ref{hyp:h1}. Let $\mathcal{A}$ be the family of admissible supersolutions from \cref{def:a-supersolution} and suppose that $w_1,w_2\in\mathcal{A}$, then $w:=\min\{w_1,w_2\}\in\mathcal{A}$.
    \end{lemma}
    \begin{proof}
        Let $\varphi\in C^2(\Omega_w\cap B_1)$ which touches $w$ from below at $x_0\in \Omega_w$, then, without loss of generality, $\varphi$ touches $w_1$ from below at $x_0$, thus $\F(D^2 \varphi(x_0),\nabla \varphi(x_0),x_0)\le f(x_0)$.
        By arbitrariness of $x_0$, we deduce that $$\F(D^2 w,\nabla w,x)\le f\quad\text{in }\Omega_w\cap B_1$$ in viscosity sense. Now given $x_0\in\partial\Omega_w\cap B_1$, if there exists $\widetilde \varphi\in C^1(B_1)$ such that $\widetilde \varphi_+$ touches either $w_1$ or $w_2$ from above at $x_0$, then $\widetilde \varphi_+$ touches $w$ from above at $x_0$; otherwise $w_1=o(|x-x_0|)$ and $w_2=o(|x-x_0|)$ as $x\to x_0$, then $w=o(|x-x_0|)$.
    \end{proof}
    
    \begin{lemma}\label{lemma:ex1}
			Suppose that $\F\in\EE$, $f\in C^0( \overline B_1)$ and $g\in {C^0}(\overline B_1,\mathbb{S}^{d-1})$ satisfies the hypothesis \ref{hyp:h1}. Let $\mathcal A$ be the family of supersolutions from \cref{def:a-supersolution} and $\underline u$ be the strict minorant from \cref{def:strictminorant}. Then, for every $w\in\mathcal{A}$ with $w\ge \underline u$ in $\overline B_1$, there exists $ \overline w\in\mathcal{A}$ such that 
   \begin{equation}\label{e:equation-interior-for-overline-w}
   \F(D^2\overline w,\nabla \overline w,x)=f\quad\text{in } \Omega_{\overline w}\cap B_1.
   \end{equation}
   Moreover 
	\begin{equation}\label{e:inequality-for-overline-w}
\underline u\le\overline w\le w\quad\text{in } \overline B_1.
   \end{equation}
   and \be\label{eq:attacca-bordo} \overline w=w\quad\text{on }\partial B_1
   \ee
		\end{lemma}
		\begin{proof}
		    We set $\Omega:=\Omega_w\cap B_1$ and \bea\mathcal{V}:=\{v\in C^0(\overline \Omega):\ \F(D^2 v,\nabla v,x)\le f \text{ in }\Omega, \ v\ge0\text{ in }\Omega,\ v= w\text{ on }\partial \Omega \}\eea and we observe that $\mathcal{V}\neq\emptyset$, since $w\in\mathcal{V}$. 
      We define $\overline w:\overline\Omega\to\R$ such that 
      $$\overline w(x):=\inf\{v(x):v\in\mathcal V\}$$ 
      and we claim that $\overline w$ is a viscosity solution to the obstacle problem in $\Omega$
      \be\label{obstacleproblem}\begin{cases}
				\F(D^2 \overline w,\nabla \overline w,x)=f \quad &\text{in }\Omega\cap\Omega_{\overline w},\\
            \F(D^2 \overline w,\nabla \overline w,x)\le f \quad &\text{in }\Omega,\\
                \overline w\ge0	\quad &\text{in } \Omega,
			\end{cases}\ee
   and that 
\be\label{e:obstacleproblem-continuity}
				\overline w\in C^0(\overline\Omega)\qquad\text{and}\qquad 
                \overline w=w \quad\text{on } \partial\Omega.
\ee
   
Let $D\subset \Omega$ a countable dense set in $\Omega$. By a diagonal argument, there is a sequence $\{v_n\}\subset\mathcal{V}$ such that \be\label{razionali}v_n(y)\to \overline w(y)\quad\text{for every}\quad y\in D\ee as $n\to+\infty$. Moreover, since the minimum of two supersolutions is still a supersolution (see \cref{rem:minimum}), we can assume that $v_n(x)$ is non-decreasing for every $x\in\overline\Omega$.	

   Let $x_0\in\Omega$, then there is $\rho>0$ such that $B_\rho(x_0)\subset \Omega$. By \cite[Theorem 2.11, Theorem 2.12]{kt19}, there exists an $L^p$-viscosity solution $h_n$ to
    $$\begin{cases}
				\F(D^2 h_n,\nabla h_n,x)=f\quad &\text{in }B_\rho(x_0)\cap\Omega_{h_n},\\
    \F(D^2 h_n,\nabla h_n,x)\le f \quad &\text{in }B_\rho(x_0),\\
	h_n\ge 0 \quad &\text{in }	B_\rho(x_0),\\
    h_n=v_n \quad &\text{on } \overline \Omega\setminus B_\rho(x_0),
			\end{cases}$$ 
   with $h_n\in C^{1,\gamma}_{loc}(B_\rho(x_0))\cap C^0(\overline B_\rho(x_0))$. Since $\F$ and $f$ are continuous, $h_n$ is also a viscosity solution in our sense (see \cite[Proposition 2.9]{ccks96}).
   In particular $\{h_n\}\subset\mathcal{V}$. By the $C_{loc}^{1,\gamma}$ estimates for $h_n$, up to a subsequence, $h_n$ uniformly converges (in the compact subsets of $B_\rho(x_0)$) to some function $h:B_\rho(x_0)\to\R$, where $h$ is a viscosity solution to 
   $$\begin{cases}
				\F(D^2 h,\nabla h,x)=f\quad &\text{in }B_\rho(x_0)\cap\Omega_{h},\\
    \F(D^2 h,\nabla h,x)\le f \quad &\text{in }B_\rho(x_0),\\
	h\ge 0 \quad &\text{in }	B_\rho(x_0).\\
			\end{cases}$$ and $h\in C^{1,\gamma}_{loc}(B_\rho(x_0))$.
   Now we prove that $h=\overline w$ in $B_\rho(x_0)$, which gives $\overline w\in C^{1,\gamma}_{loc}(\Omega)$ and concludes the proof of \eqref{obstacleproblem}. Since $\{h_n\}\subset\mathcal V$, we have that $h\ge \overline w$ in $B_\rho(x_0)$. Suppose by contradiction that $$h(x)>\overline w(x)\quad\text{for some}\quad x\in B_\rho(x_0).$$ Then, by the continuity of $h$ and the upper semicontinuity of $\overline w$, we get $$h(y)>\overline w(y)\quad\text{for some}\quad y\in D\cap B_\rho(x_0),$$  which implies that $$h_n(y)>v_n(y)$$ for $n$ large enough, by \eqref{razionali}. Then we can find some $t>0$ such that $v_n+t$ touches from above $h_n$ at some point $z\in B_\rho(x_0)$. Since $h_n(z)=v_n(z)+t>0$, we have that 
   $$\F(D^2v_n,\nabla v_n,x)\le f\quad\text{and}\quad\F(D^2h_n,\nabla h_n,x)= f$$
   in a neighborhood of $z$.
   By the maximum principle, this implies that $v_n+t\equiv h_n$ in $B_\rho(x_0)$, which is a contradiction since $v_n=h_n$ on $\partial B_\rho(x_0)$.
   This concludes the proof of \eqref{obstacleproblem}. In order to prove \eqref{e:obstacleproblem-continuity}, we consider a sequence $x_n\in \Omega$ converging to some $x_0\in\overline\Omega$. If $x_0\in \Omega$, by the regularity of the solution to the obstacle problem, we get that $\overline w(x_n)\to \overline w(x_0)$. If $x_0\in\partial\Omega\cap B_1$, using the fact that 
   $$0\le \overline w(x_n)\le w(x_n)$$
   and the continuity of $w$, we get that $\lim_{n\to+\infty}\overline w(x_n)=0=w(x_0)$. Finally, suppose that $x_0\in\partial B_1$. Then, by the continuity of $w$, we have 
   $$\lim_{n\to+\infty}\overline w(x_n)\le \lim_{n\to+\infty} w(x_n)=w(x_0),$$
   and, using a standard barrier argument from \cite[Lemma 6.5]{hanlin} or \cite[Theorem 4.16]{xavixavi} (see also \cref{prop:attacca} below) with the barrier from \cref{lemma:barrier2}, we get that the above inequality is an equality. This concludes the proof of \eqref{e:obstacleproblem-continuity}.

  We extend $\overline w=0$ outside $\overline \Omega$ and we observe that \eqref{eq:attacca-bordo} holds. Now we prove \eqref{e:inequality-for-overline-w}.
  The second inequality follows by the definition of $\overline w$. In order to prove the first inequality, we notice that since $w\ge \underline u$ and $\overline w\ge 0$, we only need to prove the first inequality in $\Omega=\Omega_w\cap B_1$. Suppose by contradiction that there is some $y\in \Omega$ such that $\underline u(y)>\overline w(y)$. Then, there is some $t>0$ and some $z\in \Omega$ such that $\overline w+t$ touches $\underline u$ from above at $z$. Since $\underline u(z)=\overline w(z)+t>0$, then  $$\F(D^2\underline u,\nabla \underline u,x)\ge f\quad\text{and}\quad\F(D^2\overline w,\nabla \overline w,x)\le f$$ in a neighborhood of $z$, which is a contradiction since $\overline w\ge \underline u$ on $\partial \Omega$.
  
  We finally prove that $\overline w\in\mathcal{A}$, i.e.~$\overline w$ satisfies the condition $iii)$ of \cref{def:a-supersolution}.  
   Let $x_0\in \partial\Omega_{\overline w}\cap B_1$, then:
   \begin{itemize}
       \item if $x_0\in \partial\Omega_{ w}$, then either there exists $\widetilde \varphi\in C^1(B_1)$ as in \cref{def:a-supersolution} or $w=o(|x-x_0|)$ as $x\to x_0$; since $\overline w\le w$, in the first case, we use that $\widetilde\varphi_+$ touches $\overline w$ from above at $x_0$, while in the second case $\overline w=o(|x-x_0|)$ as $x\to x_0$;
    \item if $x_0\in\Omega_w$, then $|\nabla \overline w(x_0)|=0$ by the $C^{1,\gamma}_{loc}$-regularity of $\overline w$; 
    thus, $\overline w=o(|x-x_0|)$.\qedhere
   \end{itemize}
        \end{proof}
		
\begin{remark}[Definition of $\mathcal A^+$]\label{remark:a} 
  We denote by $\mathcal{A}^+\subset \mathcal A$ the family of supersolutions containing the functions $\overline w\in\mathcal A$ which satisfies the interior condition \eqref{e:equation-interior-for-overline-w}. By \cref{lemma:ex1}, we can replace a supersolution $w\in\mathcal{A}_{\underline u,\phi}$ by a function $\overline w\in\mathcal{A}_{\underline u,\phi}\cap\mathcal{A}^+$, then \bea u(x)&:=\inf\{w(x):w\in\mathcal{A}_{\underline u,\phi}\}=\inf\{\overline w(x):\overline w\in\mathcal{ A}_{\underline u,\phi}\cap\mathcal{A}^+\},
  \eea where $\mathcal{A}_{\underline u,\phi}$ is the family from \eqref{def:a-u-phi} and $u$ is the function defined in \eqref{e:def-viscosity-solution-u}.
		\end{remark} In the following proposition, we prove the Lipschitz regularity for the supersolutions in the class $\mathcal{A}^+$.
		
		\begin{proposition}\label{lemma:ex2} Suppose that $\F\in\EE$, $f\in C^0( \overline B_1)$ and $g\in {C^0}(\overline B_1,\mathbb{S}^{d-1})$ satisfies the hypothesis \ref{hyp:h1}. Let $\mathcal{A}^+$ be the family of supersolutions in \cref{remark:a}, then, for every $\rho\in(0,1)$, we have $$\|\nabla w\|_{L^\infty(B_{\rho})}\le L\quad\text{for every}\quad w\in\mathcal{A}^+,$$ for some universal constant $L>0$, which also depends on $\rho$. 
		\end{proposition}

		\begin{proof}
			For simplicity, we set $\rho=1/2$. Let $x_0\in\Omega_w\cap B_{1/2}$ and $d_0:=\text{\rm dist}(x_0,\partial\Omega_w)\le\overline d_0$, for some $\overline d_0>0$ small to be chosen. We consider the rescaling $$w_{d_0}(x):=\frac{w(x_0+rx)}{d_0}.$$ Then $w_{d_0}$ is a solution in the interior \eqref{e:equation-interior-for-overline-w} and satisfies the supersolution condition in \cref{def:def-sol} on the free boundary, with operator $\F_{d_0}:=\F_{x_0,d_0}$, right-hand side $f_{d_0}:=f_{x_0,d_0}$, free boundary condition $g_{d_0}:=g_{x_0,d_0}$, as in \eqref{def:rescaling-operator}.
			We need to prove that $w_{d_0}(0)\le M$ for some constant $M>0$, which concludes the proof by classical estimates for fully nonlinear operators (see e.g. \cite{caffarellicabre}) and by Harnack inequality in \cref{prop:interiorharnack}. 
            
            We suppose by contradiction that $w_{d_0}(0)>M$ for some constant $M>0$ to be chosen. 
			By Harnack inequality in \cref{prop:interiorharnack}, there are constants $c>0$ and $C>0$ such that 
			\bea w_{d_0}&\ge 2c w_{d_0}(0)-C\|f_{d_0}\|_{L^\infty(B_1)}\ge2c w_{d_0}(0)-C\|f\|_{L^\infty(B_1)}\ge
			cw_{d_0}(0)=:\delta\quad\text{in } B_{1/2},\eea if we choose a universal $M>0$ large enough.
			Therefore, using the notations of \cref{lemma:barrier1}, if we choose $\overline d_0$ so that $\|f_{d_0}\|_{L^\infty(B_1)}\le\sigma=\sigma(\delta)$, then $\zeta:=\delta\psi$ is a subsolution to \eqref{def:def-viscosity-solution} with operator $\F_{d_0}$, right-hand side $f_{d_0}$, free boundary condition $g_{d_0}$. In particular
			$$w_{d_0}\ge \zeta\quad\text{in } B_1\setminus \overline B_{1/2}.$$
            Moreover
			$$|\nabla \zeta|\ge C_{M}>0\quad\text{in }B_1\setminus \overline B_{1/2},$$ for some $C_{M}>0$, with $C_{M}\to+\infty$ as $M\to+\infty$.
			
			Let $z_0$ the point where $d_0$ is achieved and $y_0:=\frac{z_0-x_0}{d_0}\in\partial B_1$. Since $\zeta$ touches $w_{d_0}$ from below at $y_0$, we have 
			$$|\nabla \zeta(y_0)|\le g_{d_0}\left(y_0,\frac{\nabla \zeta(y_0)}{|\nabla \zeta(y_0)|}\right)\le \|g\|_{L^\infty},$$
			which is a contradiction if $M$ is chosen large enough.
		\end{proof}
		\begin{proof}[Proof of \cref{prop:lip-and-sol}]
			The proof follows directly from \cref{lemma:ex2} and \cref{remark:a}.
		\end{proof}
        \begin{remark}[Every solution to \eqref{def:def-viscosity-solution} is Lipschitz]\label{rem:every-sol-is-lip} Let $u:B_1\to\R$ be a non-negative continuous viscosity solution to \eqref{def:def-viscosity-solution}, even if not built using the Perron's method. The same argument in \cref{lemma:ex2} shows that $u$ is Lipschitz continuous and for every $\rho>0$ there is $L>0$ such that $$\|\nabla u\|_{L^\infty(B_{\rho})}\le L.$$  Indeed, in the proof of \cref{lemma:ex2} we only used that $w\in\mathcal{A}^+$ is a solution in the interior and a supersolution on the free boundary, according to \eqref{def:def-sol}.
        \end{remark}
		As a consequence of \cref{lemma:ex2}, we obtain the following corollary.
		\begin{corollary}\label{corollary:ex0}
			Suppose that $\F\in\EE$, $f\in C^0( \overline B_1)$ and $g\in {C^0}(\overline B_1,\mathbb{S}^{d-1})$ satisfies the hypothesis \ref{hyp:h1}. Let $\mathcal{A}^+$ and $\mathcal{A}_{\underline u,\phi}$ be the families of supersolutions in \cref{remark:a} and \eqref{def:a-u-phi} respectively and let $u:\overline B_1\to\R$ be the function from \cref{t:existence}, defined in \eqref{e:def-viscosity-solution-u}. Then, for every compact subset $K\subset\subset B_1$, there is a sequence $\{w_k\}\subset\mathcal{A}_{\underline u,\phi}\cap \mathcal{A}^+$ such that $w_k$ converges to $u$ uniformly in $K$. Moreover, for every compact set $K_0\subset\subset K$, with $K_0\subset B_1\setminus \overline\Omega_u$, then $w_k$ may be taken identically zero in $K_0$.
			\begin{proof}
				The first part follows by the theorem of Ascoli-Arzel\'a. For the second part, without loss of generality, we can suppose $K_0=\overline B_{r}(x_0)$ and the sequence $\{w_{k}\}\subset\mathcal{A}_{\underline u,\phi}\cap\mathcal{A}^+ $ converges to $u$ in $K\supset \overline B_{2r}(x_0) $, with $\overline B_{2r}(x_0)\subset B_1\setminus\overline\Omega_u$. Let $$w_{k,r}(x):=\frac{w_k(x_0+rx)}{r},$$ then $w_{k,r}$ converge to $0$ in $\overline B_2$. Using the notation of \cref{lemma:barrier2}, if we choose $\delta>0$ small enough and $r>0$ small enough, there exists a function $\widetilde \zeta:=\delta\widetilde \psi$, such that $\widetilde \zeta$ is a supersolution to \eqref{def:def-viscosity-solution} in $B_2$, with operator $\F_{r}:=\F_{x_0,r}$, right-hand side $f_{r}:=f_{x_0,r}$, free boundary condition $g_{r}:=g_{x_0,r}$ as in \eqref{def:rescaling-operator}. Moreover $\widetilde \zeta\equiv\delta$ on $\partial B_2$ and $\widetilde \zeta\equiv0$ in $B_{1}$. Since for $k$ large enough $w_{k,r}<\widetilde \zeta$ on $\partial B_2$, then we can replace the sequence $w_k$ with $$ \widetilde w_k(x):=
				\begin{cases}
					\min\{ w_k(x),r\widetilde \zeta(\frac{x-x_0}{r})\}\quad&\text{if } x\in B_{2r}(x_0)\\
					w_k(x)\quad&\text{if } x\in \overline B_1\setminus \overline B_{2r}(x_0)
				\end{cases}$$ which converge to $u$ in $K$ and $\widetilde w_k\equiv0$ in $\overline B_{r}(x_0)$. Moreover $\widetilde w_k$ belong to the family $\mathcal{A}_{\underline u}^\phi$, by \cref{rem:minimum} and since $\underline u\equiv 0$ in $B_{2r}(x_0)$. Finally, we can apply \cref{lemma:ex1} to $\widetilde w_k$ to get a replacement which satisfies \eqref{e:equation-interior-for-overline-w}, as desired.
			\end{proof}
		\end{corollary}
	\subsection{Step 2. $u$ is a supersolution on the free boundary} In this subsection we prove that $u$ is a supersolution to \eqref{def:def-viscosity-solution}. In particular, we prove the following proposition.
			\begin{proposition}\label{prop:supersol}
				Suppose that $\F\in\EE$, $f\in C^0( \overline B_1)$ and $g\in {C^0}(\overline B_1,\mathbb{S}^{d-1})$ satisfies the hypothesis \ref{hyp:h1}. 
                Let $u:\overline B_1\to\R$ be the function from \cref{t:existence}, defined in \eqref{e:def-viscosity-solution-u}, then $u$ satisfies the supersolution condition on the free boundary, according to \cref{def:def-sol}.
			\end{proposition}	
        \begin{proof}
                Let us suppose by contradiction that there is a function $\varphi\in C^1(B_1)$ such that $\varphi$ touches $u$ from below at $x_0\in\partial\Omega_u\cap B_1$ and $$|\nabla \varphi(x_0)|> g\left(x_0,\frac{\nabla \varphi(x_0)}{|\nabla \varphi(x_0)|}\right).$$ 
				By continuity, we can suppose \be\label{eq:supersol-cond}|\nabla \varphi(x)|> g\left(x,\frac{\nabla \varphi(x)}{|\nabla \varphi(x)|}\right)\quad\text{for every}\quad x\in B_\rho(x_0),\ee for some $\rho>0$.    
				Without loss of generality, we can suppose that $\varphi$ touches $u$ strictly from below in $B_\rho(x_0)$ and \be\label{eq:interior-cond}\F(D^2\varphi,\nabla \varphi,x)>f\quad\text{in }B_\rho(x_0)\ee 
				By \cref{corollary:ex0}, there is a sequence $\{w_k\}\subset\mathcal{A}$ that converges to $u$ in $\overline B_\rho(x_0)$. In particular, we can suppose that $w_{k}>\varphi+\delta$ on $\partial B_\rho(x_0)$ for $k$ large enough and for some $\delta>0$. This implies that there is $c_k\in\R$ with $|c_k|<\delta$ and such that $\varphi +c_k$ touches $w_{k}$ from below in a point $x_k\in B_\rho(x_0)$. But this is a contradiction with \eqref{eq:supersol-cond} and \eqref{eq:interior-cond}.
        \end{proof}
        \subsection{Step 3. $u$ is non-degenerate}
		In this subsection we prove the non-degeneracy property of the function $u$, which is the following proposition.
		\begin{proposition}\label{prop:non-and-sol} Suppose that $\F\in\EE$, $f\in C^0( \overline B_1)$ and $g\in {C^0}(\overline B_1,\mathbb{S}^{d-1})$ satisfies the hypothesis \ref{hyp:h1}.
        Let $u:\overline B_1\to\R$ be the function from \cref{t:existence}, defined in \eqref{e:def-viscosity-solution-u}, then $u$ is non-degenerate, namely, for every $\rho\in(0,1)$, we have
        $$\sup_{B_r(x)}u\ge cr\quad\text{for every} \quad x\in \overline\Omega_u\cap B_{\rho},\quad r\in(0,r_0),$$ for some universal constants $c>0$ and $r_0>0$, which also depend on $\rho$.
		\end{proposition}        
            \begin{proof} For simplicity, we set $\rho=1/2$. By the argument in \cite[Lemma 5.1]{DESILVA-perron}, it is sufficient to prove that there are constants $c_0$ and $\overline d_0>0$ such that
				$$u(x_0)\ge c_0\,\text{\rm dist}(x_0,\partial\Omega_u)\quad\text{for every}\quad x_0\in \Omega_u\cap B_{1/2},$$ where $d_0:=\text{\rm dist}(x_0,\partial\Omega_u)\le\overline d_0.$ 
By \cref{prop:supersol}, $u$ is a supersolution to \eqref{def:def-viscosity-solution}, then we can use the strong maximum principle in \cref{lemma:maximum} to get $$\partial\Omega_{\underline u}\cap\partial\Omega_u\cap B_1=\emptyset.$$ In particular, we can take $\overline d_0>0$ small enough such that if $d_0\le \overline d_0$, then $$\text{\rm dist}(x_0, \Omega_{\underline u})\ge\frac{d_0}2,$$ 
            which implies that 
    \be\label{eq:lowboundund}\underline u\equiv0\quad\text{in }B_{d_0/2}(x_0).\ee 
				By \cref{corollary:ex0}, we can find a sequence $\{w_k\}\subset \mathcal{A}^+$ such that $w_k$ converges to $u$ in $\overline B_{3/4}$.
				We define $$ u_{d_0}(x):=\frac{u(x_0+d_0x)}{d_0} \quad\text{and}\quad  w_{k,d_0}(x):=\frac{w_k(x_0+d_0x)}{d_0}$$ and assume by contradiction that $ u_{d_0}(0)\le c_0$, for some constant $c_0$ to be chosen.
				Using the notations of \cref{lemma:barrier2}, we can find a function $\widetilde \zeta:=\delta\widetilde \psi$ which is a supersolution to \eqref{def:def-viscosity-solution} in $B_{1/2}$, with operator $\F_{d_0}:=\F_{x_0,d_0}$, right-hand side $f_{d_0}:=f_{x_0,d_0}$, free boundary condition $g_{d_0}:=g_{x_0,d_0}$ as in \eqref{def:rescaling-operator}, where $\delta>0$ is a universal constant. Moreover $\widetilde \zeta\equiv0$ in $ B_{1/4}$ and $\widetilde \zeta\equiv \delta$ on $\partial B_{1/2}$. 
				 We can apply Harnack inequality in \cref{prop:interiorharnack} to $ u_{d_0}$, then, for $k$ large enough, we get $$ w_{k,d_0}\le C(c_0+d_0\|f\|_{L^\infty})\quad\text{in }B_{1/2}.$$ Taking $c_0$ and $\overline d_0$ small enough, it follows that $ w_{k,d_0}< \widetilde \zeta$ on $\partial B_{1/2}$. If we set $$ \widetilde v_k:=
				\begin{cases}
					\min\{  w_{k,d_0},\widetilde \zeta\}\quad&\text{in } B_{1/2},\\
					 w_{k,d_0}\quad&\text{otherwise},\\
				\end{cases}$$ then $ v_k(x):=d_0\widetilde v_k(\frac{x-x_0}{d_0})$ belongs to the family $\mathcal{A}_{\underline u,\phi}$ in \eqref{def:a-u-phi}, by \cref{rem:minimum} and \eqref{eq:lowboundund}.
                This is a contradiction since $x_0\in\Omega_u$ and $v_k(0)=0$.
			\end{proof}
            \subsection{Step 4. $u$ is a subsolution on the free boundary} In this subsection we prove that $u$ is a subsolution to \eqref{def:def-viscosity-solution}. In particular, we prove the following proposition.
			\begin{proposition}\label{prop:subsol}
				Suppose that $\F\in\EE$, $f\in C^0( \overline B_1)$ and $g\in {C^0}(\overline B_1,\mathbb{S}^{d-1})$ satisfies the hypothesis \ref{hyp:h1}. 
                Let $u:\overline B_1\to\R$ be the function from \cref{t:existence}, defined in \eqref{e:def-viscosity-solution-u}, then $u$ satisfies the subsolution condition on the free boundary, according to \cref{def:def-sol}.
			\end{proposition}
			\begin{proof}
				Let us suppose by contradiction that there is a function $\varphi\in C^1(B_1)$ such that $\varphi_+$ touches $u$ from above at $x_0\in\partial\Omega_u\cap B_1$, with $|\nabla \varphi(x_0)|\not=0$ and $$|\nabla \varphi(x_0)|< g\left(x_0,\frac{\nabla \varphi(x_0)}{|\nabla \varphi(x_0)|}\right).$$ 
				By continuity, we can suppose that \be\label{eq:subsol-cond}|\nabla \varphi(x)|< g\left(x,\frac{\nabla \varphi(x)}{|\nabla \varphi(x)|}\right)\quad\text{for every}\quad x\in B_\rho(x_0).\ee 
				We set $$\alpha:=|\nabla \varphi(x_0)|>0\quad\text{and}\quad \nu:=\frac{\nabla \varphi(x_0)}{|\nabla \varphi(x_0)|}.$$ 
    Let $$u_r(x):=\frac{u(x_0+rx)}{r},$$ then $\varphi_r^+ $ touches $u_r$ from above at $0$ and $\varphi_r^+$ converges to $\varphi_0(x):=\alpha (x\cdot\nu)_+$. Since $\varphi_+=(\alpha (x-x_0)\cdot\nu+o(|x-x_0|))_+$ as $x\to x_0$, for every $\sigma>0$, there is a constant $r_0>0$ such that $\varphi_r^+$, and thus $u_r$, vanishes in the cone 
    $$\left\{\frac{x}{|x|}\cdot \nu<-\frac{\sigma}2\right\}\cap \overline B_1$$
    for every $r\in(0,r_0)$. We also notice that 
    $$ u_r\le \varphi_r^+< \varphi_0(\cdot+\sigma\nu)=:h\quad\text{in }\left\{x\cdot \nu>-\frac34\sigma\right\}\cap \partial B_1$$ 
    for every $r\in(0,r_0).$
    Moreover, by \cref{corollary:ex0}, there is a sequence $\{w_k\}\subset\mathcal{A}^+$ such that $w_k$ converges to $u$ in $\overline B_\rho(x_0)$ and for the corresponding rescalings $$w_{k,r}(x):=\frac{w_k(x_0+rx)}{r},$$ we have 
    $$w_{k,r}\equiv0\quad\text{in }\left\{x\cdot \nu\le-\frac34\sigma\right\}\cap \partial B_1.$$ 
    In particular, if $r_0$ is small enough and $k$ is large enough, we have \be\label{e:min}w_{k,r}\le h\quad\text{on }\partial B_1.\ee
				Let $D_\sigma:=\{x\cdot\nu>-\sigma+2\sigma\omega(x)\}\cap B_1 $ and $v:\overline B_1\to\R$ such that $$ \begin{cases}
					\F_r(D^2 v,\nabla v,x)=f_r\quad&\text{in }D_\sigma, \\
					v=0\quad&\text{in } B_1\setminus D_\sigma,\\
					v=h\quad&\text{on } \partial B_1,
				\end{cases}$$ 
                for some $r>0$ to be chosen, where $\F_r:=\F_{x_0,r}$ and $f_r:=f_{x_0,r}$ are as in \eqref{def:rescaling-operator} and $\omega$ is a cutoff function such that $\omega\equiv1$ in $B_{1/4}$ and $\omega\equiv0$ in $\R^d\setminus B_{1/2}$. Then, by the $C^{1,\alpha}$ regularity estimates up to the boundary (see \cite[Theorem 1.1]{silvestresirakov}) for the function $v-h$, we have $$\|v-h\|_{C^{1,\alpha}(D_\sigma)}\le C(\sigma+r).$$ In particular, if we choose $\sigma$ and $r$ small enough, $v$ is strictly monotone in the $\nu$ direction. Thus $v\ge0$ in $B_1$ and $v>0$ in $D_\sigma$. Moreover, if $g_{r}:=g_{x_0,r}$ is as in \eqref{def:rescaling-operator}, then 
                $$|\nabla v(x)|< g\left(x_0+r x,\frac{\nabla v(x)}{|\nabla v(x)|}\right)=g_{r}\left(x,\frac{\nabla v(x)}{|\nabla v(x)|}\right)\quad\text{for every}\quad x\in\partial \Omega_v\cap B_1,$$ 
                by \eqref{eq:subsol-cond} and since $\nabla v$ is close to $\nu$ on $\partial \Omega_v$. This implies that $v$ is a supersolution of \eqref{def:def-viscosity-solution} with operator $\F_r$, right-hand side $f_r$, free boundary condition $g_r$.
    Since $$\text{\rm dist}(x_0,\partial\Omega_{\underline u})>0,$$ by \cref{lemma:maximum}, then we can suppose that $\underline u\equiv0$ in $B_{r}(x_0)$.
				Then $$\widetilde w_k(x):=
				\begin{cases}
					\min\{ w_k(x),rv(\frac{x-x_0}{r})\}\quad&\text{if }x\in B_{r}(x_0)\\
					w_k(x)\quad&\text{if } x\in \overline B_1\setminus \overline B_{r}(x_0)
				\end{cases}$$ belong to the family $\mathcal{A}_{\underline u,\phi}$ in \eqref{def:a-u-phi}, by \cref{rem:minimum} and \eqref{e:min}. Moreover $\widetilde w_k\equiv0$ in a neighborhood of $x_0$, since $B_1\setminus \overline D_\sigma$ contains a neighborhood of $0$. This is a contradiction with $x_0\in\partial\Omega_u$.
			\end{proof}
			\subsection{Step 5. $u$ is continuous up to the boundary and well-defined}
            In this subsection we prove that the solution $u$ is continuous up to the boundary, with $u=\phi$ on $\partial B_1$, and $u$ is well-defined, namely $\mathcal{A}_{\underline u,\phi}\not=\emptyset$. 
            
            For what concerns the continuity of $u$ up to the boundary, we use standard barrier arguments as in \cite[Lemma 6.5]{hanlin} or \cite[Theorem 4.16]{xavixavi}.
        \begin{proposition}\label{prop:attacca}
               Suppose that $\F\in\EE$, $f\in C^0( \overline B_1)$ and $g\in {C^0}(\overline B_1,\mathbb{S}^{d-1})$ satisfies the hypothesis \ref{hyp:h1}. Let $\phi:\partial B_1\to\R$ be the boundary datum from \cref{t:existence} and let $\underline u$ be the strict minorant from \cref{def:strictminorant}, such that $\underline u\le \phi$ on $\partial B_1$.
               Let $u:\overline B_1\to\R$ be the function from \cref{t:existence}, defined in \eqref{e:def-viscosity-solution-u}, then $$u\in C^0(\overline B_1)\qquad\text{and}\qquad u=\phi\quad\text{on }\partial B_1.$$
            \end{proposition}
		\begin{proof}
		    Since $u$ is Lipschitz continuous in $B_1$ and $u=\phi$ on $\partial B_1$, in order to prove the proposition, we need to show that $\lim_{x\to x_0}u(x)=\phi(x_0)$ for every $x_0\in\partial B_1$.
            
            By the continuity of $\phi$, for every $\eps>0$ there is $\delta>0$ such that 
            $$|\phi(x)-\phi(x_0)|\le \eps\quad\text{for every}\quad x\in \partial B_1\cap B_\delta(x_0).$$
            Let $\overline x:=3/2x_0$, then, by \cref{lemma:barrier2}, there is a function $\widetilde\psi$ such that $\widetilde\psi\equiv0$ in $B_{1/2}(\overline x)$, $\widetilde\psi\equiv 1$ in $B_{1}(\overline x)$ and $$\F(D^2\widetilde\psi,\nabla\widetilde\psi,x)\le -\sigma\quad\text{in }B_{1}(\overline x)\setminus \overline B_{1/2}(\overline x),$$ for some $\sigma>0$. Since in $\overline B_1\setminus B_{\delta}(x_0)$ the function $\widetilde\psi$ is bounded by below, we have that 
    \be\label{stimaphi}|\phi(x)-\phi(x_0)|\le \eps+M\widetilde \psi \quad\text{for every}\quad x\in \partial B_1,\ee 
    for some $M>0$ large enough.
            By a similar argument, for $M$ large enough, it also holds \be\label{stimaunderline}| \underline u(x)- \underline u(x_0)|\le \eps+M\widetilde \psi \quad\text{for every}\quad x\in \overline B_1.\ee
            Up to enlarge $M$, the computation from \eqref{eq:richiama} shows that the function $\widetilde \zeta:=M\widetilde\psi $ satisfies \be\label{stima-f1}\F(D^2\widetilde\zeta,\nabla\widetilde\zeta,x)\le -C(M)\le -\|f\|_{L^\infty(B_1)}\le f\quad\text{in }B_{1}(\overline x)\setminus \overline B_{1/2}(\overline x),\ee and 
            \be\label{stima-f2}\F(-D^2\widetilde\zeta,-\nabla\widetilde\zeta,x)\ge C(M)\ge\|f\|_{L^\infty(B_1)}\ge f\quad\text{in }B_{1}(\overline x)\setminus \overline B_{1/2}(\overline x),
            \ee since $C(M)\to+\infty$ as $M\to+\infty$.
            By \eqref{stimaphi} and \eqref{stima-f1}, we have $$\phi\le \phi(x_0)+\eps+\widetilde\zeta\quad\text{on }\partial B_1\qquad\text{and}\qquad  \phi(x_0)+\eps+\widetilde\zeta\in\mathcal{A},$$ where $\mathcal{A}$ is the family of supersolutions in \cref{def:a-supersolution}. Moreover, by \eqref{stimaunderline}, we get $$\underline u\le \underline u(x_0)+\eps+\widetilde \zeta\le \phi(x_0)+\eps+\widetilde \zeta\quad\text{in }B_1,$$ then, $\phi(x_0)+\eps+\widetilde \zeta\in\mathcal{A}_{\underline u,\phi}$, where $\mathcal{A}_{\underline u,\phi}$ is the family in \eqref{def:a-u-phi}.           
            Therefore \be\label{stimattacca1}u\le \underline u(x_0)+\eps+\widetilde\zeta\quad\text{in } \overline B_1.\ee 
            On the other hand, by \eqref{stimaphi}, for every $w\in\mathcal{A}_{\underline u,\phi}$, we have 
            $$\Theta:=\phi(x_0)-\eps-\widetilde\zeta\le \phi \le w\quad\text{on } \partial B_1,$$ 
            and, for $M$ large enough, $\Theta\le w$ in $\overline B_1\setminus B_1(\overline x)$.
            We claim that \be\label{unastimafinale} \Theta\le w \quad\text{in }\overline B_1.
            \ee
            Suppose by contradiction that there is some $y\in B_1\cap( B_1(\overline x)\setminus\overline B_{1/2}(\overline x))$ such that $\Theta(y)> w(y)$. Then, there is some $t>0$ and some $z\in B_1\cap ( B_1(\overline x)\setminus \overline B_{1/2}(\overline x))$ such that $ \Theta-t$ touches $w$ from below at $z$. If $z\in B_1\setminus \overline \Omega_w$, then $|\nabla\Theta(z)|=0$, which is a contradiction. If $z\in\partial\Omega_w$, since $w$ is a supersolution on the free boundary, then $$|\nabla \Theta(z)|\le g\left(z,\frac{\nabla \Theta(z)}{|\nabla \Theta(z)|}\right)\le \|g\|_{L^\infty(B_1)},$$ which is a contradiction for $M$ large enough.
            Then $w(z)>0$ and thus $$\F(D^2\Theta,\nabla \Theta ,x)\ge f\quad\text{and}\quad\F(D^2 w,\nabla  w,x)\le f$$ in a neighborhood of $z$, by \eqref{stima-f2}, which implies that $\Theta-t\equiv w$ in the connected component of $\Omega_w\cap (B_1(\overline x)\setminus B_{1/2}(\overline x))$ where $z$ belongs. Since $\Theta$ is radial decreasing, then $$\Theta(x)-t\ge\Theta(z)-t=w(z)>0\quad\text{for every}\quad x\in B_{|z|}(\overline x),$$ then $w>0$ in $B_{|z|}(\overline x)\cap \overline B_1$. Therefore $\Theta(x_0)-t=w(x_0)$, which is a contradiction, since $$\Theta(x_0)-t=\phi(x_0)-\eps-t<\phi(x_0)\le w(x_0).$$ This proves \eqref{unastimafinale} and thus, since $u$ is the infimum of the functions in $\mathcal{A}_{\underline u,\phi}$, we obtain \be\label{stimattacca2}
            \phi(x_0)-\eps-\widetilde \zeta\le u\quad\text{in }\overline B_1
            \ee
            Then \eqref{stimattacca1} and \eqref{stimattacca2} give 
            $$ |u(x)- \phi(x_0)|\le \eps+\widetilde\zeta \quad\text{for every}\quad x\in \overline B_1.$$ 
           Passing to the limit, we have $$\lim_{x\to x_0}|u(x)-\phi(x_0)|\le \eps,$$ which implies the continuity of $u$ up to the boundary, with $u=\phi$ on $\partial B_1$, by the arbitrariness of $\eps$.
		\end{proof}	
 We finally prove that the function $u$ is well-defined.
\begin{proposition}\label{prop:construction}
        Suppose that $\F\in\EE$, $f\in C^0( \overline B_1)$ and $g\in {C^0}(\overline B_1,\mathbb{S}^{d-1})$ satisfies the hypothesis \ref{hyp:h1}.
        Let $u$ be the function from \cref{t:existence}, defined in \eqref{e:def-viscosity-solution-u}, then $u$ is well-defined, namely, if $\mathcal{A}_{\underline u,\phi}$ is the family in \eqref{def:a-u-phi}, then $$\mathcal{A}_{\underline u,\phi}\not=\emptyset.$$
        \end{proposition}
        \begin{proof} 
        Let $w(x):=C_1-C_2e^{\gamma x_d}$ for some $C_1>0$, $C_2>0$ and $\gamma>0$ to be chosen. Then, by \eqref{eq:pucci}, we have \bea \F(D^2w,\nabla w,x)&=\F(-C_2\gamma^2e^{\gamma x_d}e_d\otimes e_d,-C_2\gamma e^{\gamma x_d}e_d,x)\\&\le \mathcal{M}^+(-C_2\gamma^2e^{\gamma x_d}e_d\otimes e_d)+[\F]_{\text{Lip}(\R^d)}C_2\gamma e^{\gamma x_d}\\&
        =C_2\gamma e^{\gamma x_d}(-\lambda\gamma+[\F]_{\text{Lip}(\R^d)})\eea We choose $\gamma>\frac{[\F]_{\text{Lip}(\R^d)}}{\lambda}$ and $C_2$ large enough, so that $$\F(D^2w,\nabla w,x)\le -\|f\|_{L^\infty(B_1)}\le f\quad\text{in }B_1$$ in viscosity sense. Then we choose $C_1$ large enough such that $w$ is strictly positive, $$w\ge \underline u\quad\text{in }B_1 \qquad\text{and}\qquad w\ge \phi\quad\text{on }\partial B_1.$$ This means that $w\in\mathcal{A}_{\underline u,\phi}$ and so $\mathcal{A}_{\underline u,\phi}\not=\emptyset$, as desired.
\end{proof}

			\section{Quadratic improvement of flatness}\label{section5} In this section we prove the following quadratic improvement of flatness, following the general strategy in \cite{DeSilva:FreeBdRegularityOnePhase}.
            \begin{proposition}[Quadratic improvement of flatness]\label{t:improvement of flatness} There are universal constants $r_0>0$, $\rho>0$ and $\overline C>0$ such that the following holds. 
				Suppose that for some $r\in(0,r_0]$ the following hypotheses hold.
				\begin{enumerate}[label=\wackyenum*]
    \item\label{eq:vicinanza-F1} The fully nonlinear operator $\F\in\EE$ satisfies \ref{hyp:h2} and \bea
					\|\F(M,\xi,\cdot)&-\F(M,\eta,\cdot)\|_{L^\infty(B_1)}
					\le r|\xi-\eta| \\
					\|\F(M,\xi,\cdot)&-\F(M,\xi,0)\|_{L^\infty(B_1)}\le r^{\beta}(r+r|\xi|+|M|),
                    \eea for every  $M\in\mathcal{S}^{d\times d},\ \xi,\eta\in\R^d$.
					 \item\label{eq:vicinanza-f} The right-hand side $f$ satisfies \bea\|f-f(0)\|_{L^\infty(B_1)}\le r^{1+\beta}.\eea
					\item\label{eq:vicinanza-g1} The free boundary condition $g$ satisfies the hypothesis \ref{hyp:h1} and \bea
                      \|g(x,\nu_1)-g(0,\nu_0)-\nabla_xg(0,\nu_0)\cdot x-\nabla_\theta g(0,\nu_0)\cdot (\nu_1-\nu_0)\|_{L^\infty(B_1)}&\le r^{1+\beta}+C|\nu_1-\nu_0|^{1+\beta},\\
                    |\nabla_xg(0,\nu_1)-\nabla_xg(0,\nu_0)|&\le r |\nu_1-\nu_0|^\beta,\\ |\nabla_\theta g(0,\nu_1)-\nabla_\theta g(0,\nu_0)|&\le C |\nu_1-\nu_0|^\beta ,\\
                    \|\nabla_x g(0,\cdot)\|_{L^\infty(\mathbb{S}^{d-1})}&\le r,
                    \eea for every $\nu_0,\nu_1\in\mathbb{S}^{d-1},$ where $C>0$ is a universal constant.
					\item\label{eq:vicinanza-p} There is a quadratic polynomial \bea
					p\in\mathcal{P}(e_d,\F,f,g)\quad\text{and}\quad \|p\|_{L^\infty(B_1)}\le r,
					\eea where $\mathcal{P}(e_d,\F,f,g)$ is the family of quadratic polynomials from \eqref{def:polynomial}.
				\end{enumerate}
                Let $u:B_1\to \R$ be a non-negative continuous viscosity solution to \eqref{def:def-viscosity-solution}, with $0\in\partial\Omega_u,$ such that
					$$\left(g(0,e_d)x_d+p(x)-r^{1+\alpha}\right)_+\le u(x)\le \left(g(0,e_d)x_d+p(x)+r^{1+\alpha}\right)_+\quad\text{for every}\quad x\in B_1,$$ 
				then there are a unit vector $\nu'\in\mathbb{S}^{d-1}$ and a quadratic polynomial $p'\in\mathcal{P}(\nu',\F, f, g)$ such that
				$$\left(g(0,\nu')(x\cdot\nu')+\rho p'(x)-(r\rho)^{1+\alpha}\right)_+\le u_{\rho}(x)\le \left(g(0,\nu')(x\cdot\nu')+\rho p'(x)+(r\rho)^{1+\alpha}\right)_+$$ for every $x\in B_1$, where $u_\rho(x):=\frac{u(\rho x)}{\rho}$.
				Moreover, the following estimates hold
				\be\label{eq:stime}|\nu'-e_d|\le \overline Cr^{1+\alpha},\quad \|p'-p\|_{L^\infty(B_1)}\le \overline Cr^{1+\alpha}.\ee 
			\end{proposition}

			\subsection{Some estimates.}
			In this subsection we prove the following lemmas, which provide useful estimates for the proof of \cref{t:improvement of flatness}.
			\begin{lemma}\label{lemma:another-estimate}
                Suppose that the hypotheses \ref{eq:vicinanza-F1}, \ref{eq:vicinanza-f}, \ref{eq:vicinanza-p} hold.
                Let $u:B_1\to \R$ be a non-negative continuous viscosity solution to \eqref{def:def-viscosity-solution}, then $$|\F(D^2p,g(0,e_d)e_d+\nabla p,x)-f(x)|\le Cr^{1+\beta}\quad\text{for every}\quad x\in B_1,$$ for some universal constant $C>0$.
			\end{lemma}
			\begin{proof} First notice that by \ref{eq:vicinanza-F1} and \ref{eq:vicinanza-p}, we have
				\bea |\F(D^2p,g(0,e_d)e_d+\nabla p,x)-\F(D^2p,g(0,e_d)e_d+\nabla p,0)|&\le r^{\beta}(r+r|g(0,e_d)e_d+\nabla p|\\&\qquad+|D^2p|)\\&\le Cr^{1+\beta}.\eea
				By \ref{eq:vicinanza-p} we obtain $\mathcal F(D^2p,g(0,e_d)e_d,0)=f(0)$, then we have 
				\begin{align*}|\F(D^2p,g(0,e_d)e_d+\nabla p,x)-f(x)|&=|\F(D^2p,g(0,e_d)e_d+\nabla p,x)-\F(D^2p,g(0,e_d)e_d+\nabla p,0)\\&\qquad+\F(D^2p,g(0,e_d)e_d+\nabla p,0)-f(x)|\\&=|\F(D^2p,g(0,e_d)e_d+\nabla p,x)-\F(D^2p,g(0,e_d)e_d+\nabla p,0)\\&\qquad+\F(D^2p,g(0,e_d)e_d+\nabla p,0)-f(x)
					\\&\qquad\qquad+f(0)-\F(D^2p,g(0,e_d)e_d,0)|\\& \le
					|\F(D^2p,g(0,e_d)e_d+\nabla p,x)-\F(D^2p,g(0,e_d)e_d+\nabla p,0)|\\&\qquad+
					|\F(D^2p,g(0,e_d)e_d+\nabla p,0)-\F(D^2p,g(0,e_d)e_d,0)|\\&\qquad\qquad+|f(x)-f(0)|\\&\le Cr^{1+\beta},    
				\end{align*} for every $x\in B_1$, where we used \ref{eq:vicinanza-F1} and \ref{eq:vicinanza-f}.
			\end{proof}
			
			\begin{lemma}\label{rem:bound}
				Suppose that the hypotheses \ref{eq:vicinanza-g1}, \ref{eq:vicinanza-p} hold. Let $u:B_1\to \R$ be a non-negative continuous viscosity solution to \eqref{def:def-viscosity-solution} such that $$\left(g(0,e_d)x_d+p(x)-r^{1+\alpha}\right)_+\le u(x)\le\left(g(0,e_d)x_d+p(x)+r^{1+\alpha}\right)_+\quad\text{for every}\quad x\in B_1.$$ 
				Then $$|\nabla p(x_0)\cdot \tau (e_d) -\nabla_x g(0,e_d)\cdot x_0|\le Cr^{2}\quad\text{for every}\quad x_0\in\partial\Omega_u\cap B_1,$$ for some universal constant $C>0$.
			\end{lemma}
			\begin{proof} Notice that the flatness hypothesis can be rewritten as $$|u(x)-g(0,e_d)x_d-p(x)|\le r^{1+\alpha}\quad\text{for every}\quad x\in\overline \Omega_u\cap B_1. $$ Let $x_0\in\partial\Omega_u\cap B_1$ and $\overline x_0:=x_0-(x_0)_d\in B_1'$ be the projection of $x_0$ on $B_1'$.
				Since $u(x_0)=0$, by \ref{hyp:h1} and \ref{eq:vicinanza-p} we have that $$|x_0-\overline x_0|=|(x_0)_d|\le \frac{2r}{g(0,e_d)}\le Cr,$$ for some universal constant $C>0$.
				Then
				\begin{align*}|\nabla p(x_0)\cdot \tau(e_d)-\nabla_x g(0,e_d)\cdot x_0|&\le |\nabla p(x_0)\cdot \tau(e_d)-\nabla p(\overline x_0)\cdot \tau(e_d)|\\&\qquad+|\nabla p(\overline x_0)\cdot \tau(e_d)-\nabla_x g(0,e_d)\cdot x_0|\\
					&=|\nabla p(x_0)\cdot \tau(e_d)-\nabla p(\overline x_0)\cdot \tau(e_d)|\\&\qquad+|\nabla_xg(0,e_d)\cdot \overline x_0-\nabla_x g(0,e_d)\cdot x_0|
					\\&\le C\|D^2p\|_{L^\infty(B_1)} |x_0-\overline x_0|+|\nabla_x g(0,e_d)||x_0-\overline x_0|
					\\&\le Cr^{2},\end{align*} where we used \ref{eq:vicinanza-g1} and \ref{eq:vicinanza-p}.
			\end{proof}
			\begin{lemma}\label{lemma:definitive} 
           Suppose that the hypotheses \ref{eq:vicinanza-g1}, \ref{eq:vicinanza-p} hold.
            Let $u:B_1\to \R$ be a non-negative continuous viscosity solution to \eqref{def:def-viscosity-solution} such that $$\left(g(0,e_d)x_d+p(x)-r^{1+\alpha}\right)_+\le u(x)\le\left(g(0,e_d)x_d+p(x)+r^{1+\alpha}\right)_+\quad\text{for every}\quad x\in B_1.$$ Let $$v(x):=g(0,e_d)x_d+p(x)+r^{1+\alpha}q(x),$$ for some function $q\in C^1(B_1)$, then, for every $x_0\in \partial\Omega_u\cap B_1$, we have
            \bea g\left(x_0,\frac{\nabla v(x_0)}{|\nabla v(x_0)|}\right)&=g(0,e_d)+\partial_{x_d}p(x_0)+r^{1+\alpha}\partial_{x_d}q(x_0)-r^{1+\alpha}\nabla q(x_0)\cdot \tau(e_d,g)
           +O(r^{1+\beta}),\eea where $\tau(\cdot,\cdot)$ is defined in \eqref{def:tau}.
			\end{lemma}
			\begin{proof}
				We observe that, for every $j=1,\ldots,d-1$, we have $$\frac{\partial_{x_j} v}{|\nabla v|}=\frac{\partial_{x_j}p+r^{1+\alpha}\partial_{x_j}q}{\sqrt{(g(0,e_d)+\partial_{x_d}p+r^{1+\alpha}\partial_{x_d}q)^2+(\nabla_{x'}p+r^{1+\alpha}\nabla_{x'}q)^2}}=\frac{\partial_{x_j}p+r^{1+\alpha}\partial_{x_j}q}{g(0,e_d)}+O(r^2)$$ and $$\frac{\partial_{x_d} v}{|\nabla v|}=\frac{g(0,e_d)+\partial_{x_d}p+r^{1+\alpha}\partial_{x_d}q}{\sqrt{(g(0,e_d)+\partial_{x_d}p+r^{1+\alpha}\partial_{x_d}q)^2+(\nabla_{x'}p+r^{1+\alpha}\nabla_{x'}q)^2}}=1+O(r^2),$$ then  
				\bea & \frac{\nabla v}{|\nabla v|}-e_d=      
				\frac{1}{g(0,e_d)} (\nabla p+r^{1+\alpha}\nabla q-e_d\partial_{x_d}p-e_dr^{1+\alpha}\partial_{x_d}q)+O(r^{2}).
				\eea By \ref{eq:vicinanza-g1}, this implies that
                \bea g\left(x,\frac{\nabla v(x)}{|\nabla v(x)|}\right)&=g(0,e_d)+\nabla_xg(0,e_d)\cdot x+\\&\qquad+ \left(\frac{\nabla_\theta g(0,e_d)}{g(0,e_d)} \cdot(\nabla p(x)+r^{1+\alpha}\nabla q(x)-e_d\partial_{x_d}p(x)-e_dr^{1+\alpha}\partial_{x_d}q(x)
				\right)\\&\qquad\qquad+O(r^{1+\beta}),\eea 
                from which we get \bea g\left(x,\frac{\nabla v(x)}{|\nabla v(x)|}\right)&=g(0,e_d)+\nabla_xg(0,e_d)\cdot x+\partial_{x_d}p(x)-\nabla p(x)\cdot \tau(e_d,g)+r^{1+\alpha}\partial_{x_d}q(x)\\&\qquad-r^{1+\alpha}\nabla q(x)\cdot \tau(e_d,g)
           +O(r^{1+\beta}).\eea
            This is exactly the thesis if $x=x_0\in\partial\Omega_u\cap B_1$, by \cref{rem:bound}.
			\end{proof}
			\subsection{Contradiction assumptions}
			In order to prove \cref{t:improvement of flatness} suppose by contradiction that the hypotheses \ref{eq:vicinanza-F1}, \ref{eq:vicinanza-f}, \ref{eq:vicinanza-g1}, \ref{eq:vicinanza-p} hold for the sequences $\F_n$, $f_n$, $g_n$, $p_n$, $r_n$, with $r_n\to0^+$ and we suppose that there are $u_n:B_1\to \R$ non-negative continuous viscosity solutions to 
			\begin{equation}\label{def:def-viscosity-sol-n}\begin{cases}
					\F_n(D^2 u_n,\nabla u_n,x)=f_n(x) \quad&\text{in } \Omega_{u_n}\cap B_1,\\
					|\nabla u_n|=g_n(x,\nu) \quad&\text{on } \partial\Omega_{u_n}\cap B_1,\\
				\end{cases} 
                \end{equation} with $0\in\partial\Omega_{u_n}$,
			 such that $$\left(g_n(0,e_d)x_d+p_n(x)-r_n^{1+\alpha}\right)_+\le u(x)\le \left(g_n(0,e_d)x_d+p_n(x)+r_n^{1+\alpha}\right)_+\quad\text{for every}\quad x\in B_1,$$ but the conclusion of \cref{t:improvement of flatness} does not hold.
			We divide the proof of \cref{t:improvement of flatness} in three steps.
			\subsection{Step 1. Compactness}\label{section-compactness} In this subsection we prove that the linearized sequence 
			\be\label{e:def-of-u-tilde_k} \widetilde u_n(x):=\frac{u_n(x)-g_n(0,e_d)x_d-p_n(x)}{r_n^{1+\alpha}}\quad\text{for every}\quad x\in B_1\cap \overline\Omega_{u_n},\ee 
			converges uniformly in the compact subsets of $B_{1/2}^+\cup B_{1/2}'$ to some function $\widetilde u:B_{1/2}^+\cup B_{1/2}'\to[-1,1]$ and the graphs of $\widetilde u_n:B_1\cap \overline\Omega_{u_n}\to\R$ converge in the Hausdorff distance to the graph of $\widetilde u$.
			We deduce the compactness of $\widetilde u_n$ from the following partial Harnack inequality. 
			\begin{proposition}[Partial Harnack inequality]\label{prop:partial-harnack}
				There are universal constants $r_0>0$, $\delta>0$ and $c\in(0,1)$ such that the following holds. Suppose that the hypotheses \ref{eq:vicinanza-F1}, \ref{eq:vicinanza-f}, \ref{eq:vicinanza-g1}, \ref{eq:vicinanza-p} hold.
                Let $u:B_1\to \R$ be a non-negative continuous viscosity solution to \eqref{def:def-viscosity-solution} such that $$(g(0,e_d)x_d+p(x)+a_0)_+\le u(x)\le (g(0,e_d)x_d+p(x)+b_0)_+\quad\text{for every}\quad x\in B_1,$$ with $r^{1+\alpha}:=b_0-a_0\le r_0^{1+\alpha}$.
				Then there are $a_1$ and $b_1$, with $a_0\le a_1\le b_1\le b_0$, such that $$(g(0,e_d)x_d+p(x)+a_1)_+\le u(x)\le (g(0,e_d)x_d +p(x)+b_1)_+\quad\text{for every}\quad x\in B_{\delta},$$ with the estimate $b_1-a_1\le (1-c)(b_0-a_0).$
			\end{proposition}
			\begin{proof}
				Let $\delta\in(0,1/10)$ to be chosen later.
				Notice that it is sufficient to prove the partial Harnack inequality in the case $|a_0|< 2\delta$, by the interior Harnack inequality in \cref{prop:interiorharnack}. 
				Let $\overline x:=1/5 {e_d}$, $\rho:=1/5+3\delta$ and
				$$\psi(x):=\begin{cases}
					1 \quad&\text{if }x\in B_{1/20}(\overline x),\\
					\overline c\left(|x-\overline x|^{-\gamma}-\rho^{-\gamma}\right)\quad&\text{if }x\in B_{\rho}(\overline x)\setminus \overline B_{1/20}(\overline x),\\
					0\quad&\text{if }x \in \R^d\setminus \overline B_{\rho}(\overline x),
				\end{cases}$$ 
				with $\overline c>0$ chosen in such a way to have that $\psi$ is continuous. 
				Notice that, if $x\in B_{\rho}(\overline x)\setminus \overline B_{1/20}(\overline x)$, in an appropriate system of coordinates, we have
				$$D^2\psi(x)=\overline c\gamma |x-\overline x|^{-\gamma-2}\text{diag}(\gamma+1,-1,\ldots,-1),$$ then, by \eqref{eq:pucci}, we have that \be\label{eq:subsol}\mathcal{M}^-(D^2\psi)=\overline c\gamma |x-\overline x|^{-\gamma-2}\left(\lambda(\gamma+1)-\Lambda(d-1)\right)\ge c_0>0,\ee for some universal constant $c_0>0$, if we choose $\gamma>\max\{\frac{\Lambda(d-1)}\lambda-1,0\}$.
				Moreover, if $\tau(\cdot,\cdot)$ is as in \eqref{def:tau}, by \ref{hyp:h1} we can choose a universal $\eta>0$ such that 
				\bea \nabla \psi(x)\cdot \tau(e_d,g)&=\overline c\gamma|x-\overline x|^{-\gamma-2}(\overline x-x)\cdot \left(e_d+\frac{\nabla_{\theta'}g(0,e_d)}{g(0,e_d)}\right)\\
				&= \overline c\gamma|x-\overline x|^{-\gamma-2}\left(\frac15-x_d-x\cdot\frac{\nabla_{\theta'}g(0,e_d)}{g(0,e_d)}\right)\\&\ge c_0
				\eea for every $x\in B_\eta$, up to choose $c_0$ small enough. In particular, we choose a universal $\delta>0$ such that $$(B_{\rho}(\overline x)\setminus \overline B_{1/20}(\overline x))\cap \{x_d<2\delta\}\subset B_\eta,$$
				Then
				\be\label{e:stima-per-w}\nabla \psi(x)\cdot \tau(e_d,g)\ge c_0\quad\text{for every}\quad x\in(B_{\rho}(\overline x)\setminus \overline B_{1/20}(\overline x))\cap \{x_d<2\delta\}.\ee 
				We consider the polynomial $$P(x):=g(0,e_d)x_d+p(x)+a_0$$ and suppose that $u(\overline x)\ge P(\overline x)+r^{1+\alpha}/2$ (the other case is similar). Then, since $|a_0|\le 1/10$ and by \ref{eq:vicinanza-F1}, we have that 
				$u-P$ is a solution to a fully nonlinear elliptic equation in $ B_{1/10}(\overline x)$, with operator $$\overline \F(M,\xi,x):=\F(D^2P+M,\nabla P+\xi,x)-\F(D^2P,\xi,x)\in\EE$$ and right-hand side \begin{align*}\overline f(x)&=\overline \F(D^2u-D^2P,\nabla u-\nabla P,x)=\F(D^2u,\nabla u,x)-\F(D^2p,g(0,e_d)e_d+\nabla p,x)\\&=f(x)-\F(D^2p,g(0,e_d)e_d+\nabla p,x)
				\end{align*}
				by \ref{eq:vicinanza-p}. Then we have that $$\|\overline f\|_{L^\infty(B_{1/10}(\overline x))}\le Cr^{1+\beta},$$ by \cref{lemma:another-estimate}.
				We deduce that $$u-P\ge \widetilde cr^{1+\alpha}-C\|\overline f\|_{L^\infty(B_{1/10}(\overline x))}\ge cr^{1+\alpha}\quad\text{in } B_{1/20}(\overline x),$$ by the interior Harnack inequality in \cref{prop:interiorharnack}. 

				We define the competitors $$v_{t}(x):=P(x)+ c r^{1+\alpha} \psi(x)- cr^{1+\alpha}+ cr^{1+\alpha} t\quad\text{for every} \quad t\in[0,1),$$
				and we suppose by contradiction that there is $t\in[0,1)$ and $x_0\in B_{\rho}(\overline x)\setminus \overline B_{1/20}(\overline x)$ such that $v_{t}$ touches from below $u$ in $x_0$. 
				Suppose that $x_0\in\Omega_u$. Notice that if $r_0>0$ is small enough, by \cref{lemma:another-estimate} and \ref{eq:vicinanza-F1}, we get \begin{align*}\F(D^2 v_t(x_0),\nabla v_t(x_0),x_0)&=\F(D^2 P+cr^{1+\alpha}D^2\psi(x_0),\nabla P+cr^{1+\alpha}\nabla \psi(x_0),x_0)\\&
					\ge \F(D^2P+cr^{1+\alpha}D^2\psi(x_0),\nabla P,x_0)-Cr^{2+\alpha}
					\\&\ge \F(D^2P,\nabla P,x_0)+cr^{1+\alpha}\mathcal{M}^-(D^2\psi(x_0))-Cr^{2+\alpha}
					\\&
					= \F(D^2p,g(0,e_d)e_d+\nabla p,x_0)+cr^{1+\alpha}c_0-Cr^{2+\alpha}
					\\&\ge f(x_0)-Cr^{1+\beta}+cr^{1+\alpha}c_0-Cr^{2+\alpha}\\
					&>f(x_0)\end{align*} where we used \eqref{eq:pucci} and \eqref{eq:subsol}. This is a contradiction with \cref{def:def-sol}. Then $x_0\in\partial\Omega_u$ and $x_0\in(B_{\rho}(\overline x)\setminus \overline B_{1/20}(\overline x))\cap \{x_d<2\delta\}$, since $|a_0|< 2\delta$.
				Using \eqref{e:stima-per-w} and \cref{lemma:definitive}, we have that 
                \bea \partial_{x_d}v_t(x_0)&=g(0,e_d)+\partial_{x_d}p(x_0)+cr^{1+\alpha}\partial_{x_d}\psi(x_0)\\&=g(0,e_d)+\partial_{x_d}p(x_0)+cr^{1+\alpha}\partial_{x_d}\psi(x_0)- cr^{1+\alpha}\nabla \psi(x_0)\cdot\tau(e_d,g)\\&\qquad+cr^{1+\alpha}\nabla \psi(x_0)\cdot\tau(e_d,g)
                \\&\ge                g\left(x_0,\frac{\nabla v_t(x_0)}{\nabla v_t(x_0)|}\right)-Cr^{1+\beta}+cr^{1+\alpha}\nabla \psi(x_0)\cdot \tau(e_d,g)\\&\ge g\left(x_0,\frac{\nabla v_t(x_0)}{\nabla v_t(x_0)|}\right)-Cr^{1+\beta}+cc_0r^{1+\alpha}\\&>g\left(x_0,\frac{\nabla v_t(x_0)}{\nabla v_t(x_0)|}\right),
                \eea
			for $r_0>0$ small enough.
			This is a contradiction with \cref{def:def-sol}.
			
			It follows that $u\ge v_{1}=P+cr^{1+\alpha}\psi$ in $B_1$, which concludes the proof, since $\psi$ is bounded from below in $ B_{\delta}$.
		\end{proof}
		By the partial Harnack inequality we deduce the convergence of the sequence $\widetilde u_n$ defined in \eqref{e:def-of-u-tilde_k}, which is the following corollary.
        
		\begin{corollary}[Convergence of $\widetilde u_n$]\label{corollary:convergence-uniform-hausdorff}
        Suppose that the hypotheses \ref{eq:vicinanza-F1}, \ref{eq:vicinanza-f}, \ref{eq:vicinanza-g1}, \ref{eq:vicinanza-p} hold for the sequences $\F_n$, $f_n$, $g_n$, $p_n$, $r_n$, with $r_n\to0^+$.
        Let $u_n:B_1\to \R$ be non-negative continuous viscosity solutions to \eqref{def:def-viscosity-sol-n} such that $$\left(g_n(0,e_d)x_d+p_n(x)-r_n^{1+\alpha}\right)_+\le u_n(x)\le \left(g_n(0,e_d)x_d+p_n(x)+r_n^{1+\alpha}\right)_+\quad\text{for every}\quad x\in B_1.$$ 
			Let $\widetilde u_n$ the linearized sequence in \eqref{e:def-of-u-tilde_k},
			then there is an H\"older function $ \widetilde u:B_{1/2}^+\cup B_{1/2}'\to[-1,1]$ such that, up to subsequences, the following holds.
			\begin{itemize}
				\item For any $\delta>0$, the sequence $\widetilde u_n$ converges to $\widetilde u$ uniformly in $\{x_d\ge\delta\}$.\smallskip
				\item The sequence of graphs of $\widetilde u_n$ $$\Gamma_n:=\{(x,\widetilde u_n(x)): x\in B_{1/2}\cap\overline\Omega_{u_n}\}$$ converges in Hausdorff distance to the graph of $\widetilde u$ $$\Gamma:=\{(x,\widetilde u(x)): x\in B_{1/2}^+\cup B_{1/2}'\}.$$
			\end{itemize}
		\end{corollary}
		\begin{proof}
			For what concerns the first point, by Ascoli-Arzel\'a theorem it is sufficient to prove that the following H\"older-type estimate
            \be\label{hold-type-est}|\widetilde u_n(x)-\widetilde u_n(x_0)|\le C|x-x_0|^\gamma\quad\text{for every}\quad x\in \overline\Omega_{u_n}\cap(B_{1/2}(x_0)\setminus \overline B_{r^{1+\alpha}/r^{1+\alpha}_0}(x_0)),\ee for every $x_0\in \overline \Omega_{u_n}\cap B_{1/2}$ and for some universal constants $C>0$ and $\gamma\in(0,1)$.
            
            Let $r_0>0$, $\delta>0$ and $c\in(0,1)$ as in \cref{prop:partial-harnack}.	
            We fix $n\in\N$ and $x_0\in\overline\Omega_{u_n}\cap  B_{1/2} $. We set $\rho_j:=\frac12\delta^j$ for every $j\in\N$ and $m_n\in\N$ the larger number such that $r_n^{1+\alpha}\le \frac12r_0^{1+\alpha}\rho_{m_n}$. Then we can apply the partial Harnack in \cref{prop:partial-harnack} repeatedly to get that for every $j=0,\ldots,m_n$ $$(g_n(0,e_d)x_d+p_n(x)+a_{n,j})_+\le u_n(x)\le (g_n(0,e_d)x_d+p_n(x)+b_{n,j})_+\quad\text{for every}\quad x\in B_{\rho_j}(x_0),$$ for sequences $$-r_n^{1+\alpha}=:a_{n,0}\le\ldots\le  a_{n,m}\le b_{n,m}\le\ldots\le b_{n,0}:=r_n^{1+\alpha}$$ such that $b_{n,j}-a_{n,j}\le(1-c)^j(b_{n,0}-a_{n,0})$. Therefore, for every $j=0,\ldots,m_n$, we have that $$|\widetilde u_n(x)-\widetilde u_n(x_0)|\le 2(1-c)^j\quad\text{for every}\quad x\in \overline\Omega_{u_n}\cap B_{\rho_j}(x_0).$$ Then \eqref{hold-type-est} follows choosing $\rho_{j+1}<|x-x_0|\le \rho_j$.
		
        Once the first point was proved, the second point follows by standard arguments (see e.g. \cite{DeSilva:FreeBdRegularityOnePhase,Velichkov:RegularityOnePhaseFreeBd}).
        \end{proof}
		
		\subsection{Step 2. Linearized problem.}\label{section-linearized-problem} In this subsection we prove that the limit function $\widetilde u$ defined in \cref{corollary:convergence-uniform-hausdorff} is a solution to a fully nonlinear elliptic problem with oblique boundary condition.
		\begin{proposition}[Linearized problem]\label{lemma:linearized-problem} 
        Suppose that the hypotheses \ref{eq:vicinanza-F1}, \ref{eq:vicinanza-f}, \ref{eq:vicinanza-g1}, \ref{eq:vicinanza-p} hold for the sequences $\F_n$, $f_n$, $g_n$, $p_n$, $r_n$, with $r_n\to0^+$.
        Let $u_n:B_1\to \R$ be non-negative continuous viscosity solutions to \eqref{def:def-viscosity-sol-n}, with $0\in\partial\Omega_{u_n}$, such that $$\left(g_n(0,e_d)x_d+p_n(x)-r_n^{1+\alpha}\right)_+\le u_n(x)\le \left(g_n(0,e_d)x_d+p_n(x)+r_n^{1+\alpha}\right)_+\quad\text{for every}\quad x\in B_1.$$ 
			Then the limit function $\widetilde u$ defined in \cref{corollary:convergence-uniform-hausdorff} is a viscosity solution to the following problem \be\label{lineariz}\begin{cases}
				\widetilde \F(D^2\widetilde u)=0 \quad&\text{in }B_{1/2}^+,\\
				\nabla\widetilde u\cdot \tau=0\quad&\text{on }B_{1/2}',
			\end{cases}\ee
			where $\widetilde \F\in\EE$ satisfies \ref{hyp:h2} and $\tau\in\R^d$ satisfies $\tau\cdot e_d= 1$ and $|\tau|\le C$, where $C>0$ is a universal constant.		
		\end{proposition}
		\begin{proof}
		We define \begin{equation}\label{ftilde}\begin{aligned}\widetilde \F_n(M,\xi,x)&:=\frac1{r_n^{1+\alpha}}\Big(\F_n(D^2p_n+r_n^{1+\alpha}M,g_n(0,e_d)e_d+\nabla p_n+r_n^{1+\alpha}\xi,x)\\&\qquad-\F_n(D^2p_n,g_n(0,e_d)e_d+\nabla p_n,x)\Big),
				\end{aligned}
			\end{equation}
			and, up to a subsequence, we set
			\be\label{Finfty}\widetilde \F_\infty(M):=\lim_{n\to+\infty}\widetilde \F_n(M,0,0)\ee 
			where the limit is uniform on all compact subsets of $\mathcal{S}^{d\times d}$. Then, $\widetilde \F_\infty\in\EE$ satisfies \ref{hyp:h2}.
Let $\tau(\cdot,\cdot)$ as in \eqref{def:tau}, we define, up to a subsequence, \be\label{tau(e_d)}\tau(e_d):=\lim_{n\to+\infty}\tau(e_d,g_n),\ee which satisfies $\tau(e_d)\cdot e_d=1$ and $ |\tau(e_d)|\le C$ by \ref{hyp:h1}. We prove that $\widetilde u$ is a viscosity solution to \eqref{lineariz} with \be\label{def2}\widetilde \F:=\widetilde \F_\infty\quad\text{and}\quad\tau:=\tau(e_d).\ee 
            
            Let $\widetilde P$ be a polynomial touching $\widetilde u$ from below
			in a point $x_0\in B_{1/2}^+\cup B'_{1/2}$ (the case when $\widetilde u$ is touched by above is similar). 
            By \cref{corollary:convergence-uniform-hausdorff}, there exists a sequence of points $x_n\in\overline\Omega_{u_n}$ such that $\widetilde P$ touches $\widetilde u_n$ from below in $x_n$ and $x_n\to x_0$ as $n\to+\infty$.
			We define the polynomial \bea\label{pol:P}P(x):=g_n(0,e_d)x_d+p_n(x)+r_n^{1+\alpha}\widetilde P(x)\eea which touches from below $ u_n$ in $x_n$.

            \smallskip
			\textit{Step 1: interior condition.} We first prove that $\widetilde \F(D^2\widetilde u)\le0$ in $B_{1/2}^+$ in viscosity sense. Suppose that $x_0\in B_{1/2}^+$, then, for $n$ large enough, $x_n\in\Omega_{u_n}$ and thus \begin{align*}f(x_n)&\ge 
				\F_n(D^2 P(x_n),\nabla P(x_n),x_n)\\&=
			\F_n(D^2p_n+r_{n}^{1+\alpha}D^2 \widetilde P(x_n),g_n(0,e_d)e_d+\nabla p_n(x_n)+r_n^{1+\alpha}\nabla \widetilde P(x_n),x_n)\\&=r_n^{1+\alpha}\widetilde \F_n(D^2\widetilde P(x_n),\nabla \widetilde P(x_n),x_n)+\F_n(D^2p_n,g_n(0,e_d)e_d+\nabla p_n(x_n),x_n).\end{align*}
			It follows that $$r_{n}^{1+\alpha} \widetilde \F_n(D^2 \widetilde P(x_n),\nabla \widetilde P(x_n),x_n)\le f_n(x_n)-\F_n(D^2p_n,g_n(0,e_d)e_d+\nabla p_n(x_n),x_n)\le Cr_n^{1+\beta},$$
			by \cref{lemma:another-estimate}. This implies that $$\widetilde \F_n(D^2 \widetilde P(x_n),\nabla \widetilde P(x_n),x_n)\le Cr_n^{\beta-\alpha}$$
			and thus \begin{align*}\F_n(D^2\widetilde P(x_n),0, 0)&=\F_n(D^2\widetilde P(x_n),0, 0)-\F_n(D^2\widetilde P(x_n),\nabla \widetilde P(x_n), 0)\\&\qquad+\F_n(D^2\widetilde P(x_n),\nabla \widetilde P(x_n), 0)-\F_n(D^2\widetilde P(x_n),\nabla \widetilde P(x_n), x_n)\\&\qquad\qquad+\F_n(D^2\widetilde P(x_n),\nabla \widetilde P(x_n), x_n)
				\\&\le Cr_n+Cr_n^{\beta}+Cr_n^{\beta-\alpha}
			\end{align*} by \ref{eq:vicinanza-F1}.
			Therefore $$\widetilde \F_\infty(D^2 \widetilde P(x_0),0,0)\le0,$$ passing to the limit as $n\to+\infty$. This means that $\widetilde \F(D^2\widetilde u)\le0$ in $B_{1/2}^+$.
            \smallskip
            
			\textit{Step 2: boundary condition.} Now we prove that $\nabla \widetilde u\cdot \tau\le0$ on $B_{1/2}'$ in viscosity sense. Suppose that $x_0\in B_{1/2}'$. Without loss of generality, we can suppose that $x_n\in\partial\Omega_{u_n}\cap B_{1/2}$ for $n$ large enough.
			Then \bea\begin{aligned}|\nabla P(x_n)|^2&=(g_n(0,e_d)+\partial_{x_d}p_n(x_n)+r_n^{1+\alpha}\partial_{x_d} \widetilde P(x_n))^2+\sum_{j=1}^{d-1}(\partial_{x_j} p_n(x_n)+r_n^{1+\alpha}\partial_{x_j} \widetilde P(x_n))^2,\end{aligned}\eea and we deduce that \begin{align*}|\nabla P(x_n)|^2&=g_n(0,e_d)^2+2g_n(0,e_d)\partial_{x_d} p_n(x_n)+2r_n^{1+\alpha}g_n(0,e_d)\partial_{x_d} \widetilde P(x_n)+O(r_n^2),\end{align*}
			where we used \ref{eq:vicinanza-p}. Then $$|\nabla P(x_n)|=g_n(0,e_d)+\partial_{x_d} p_n(x_n)+r_n^{1+\alpha}\partial_{x_d} \widetilde P(x_n)+O(r_n^2).$$
			Since $P$ touches $u_n$ from below in $x_n$, then $$|\nabla P(x_n)|\le g_n\left(x_n,\frac{\nabla P(x_n)}{|\nabla P(x_n)|}\right)$$ and so $$g_n\left(x_n,\frac{\nabla P(x_n)}{|\nabla P(x_n)|}\right)\ge g_n(0,e_d)+\partial_{x_d}p_n(x_n)+r_n^{1+\alpha}\partial_{x_d} \widetilde P(x_n)+O(r_n^2).$$
            Then, by \cref{lemma:definitive}, we get
            $$r_n^{1+\alpha}\nabla \widetilde P(x_n)\cdot \tau(e_d,g_n)+O(r_n^{1+\beta})\le 0,$$
			Therefore, dividing by $r_n^{1+\alpha}$ and sending $n\to+\infty$, we obtain			$$\nabla\widetilde P(x_0)\cdot \tau(e_d)\le 0.$$ This means that $\nabla\widetilde u\cdot \tau\le 0$, as desired.
		\end{proof}
		\subsection{Step 3. Contradiction.}\label{section-contradiction} In this section we use the result in \cref{lemma:linearized-problem} to obtain a contradiction and thus to conclude the improvement of flatness in \cref{t:improvement of flatness}. We first prove the following technical lemma.
		\begin{lemma}
			\label{lemma:stima_vicino_p_p'}
			Let $\nu'\in \mathbb{S}^{d-1}$ and $\tau,\tau',\xi,\xi'\in\R^d$. 
			Let $p$ be a quadratic polynomial such that
			\begin{equation}
				\label{e:boundary_cond_p_W}
				\nabla p \cdot \tau = \xi\cdot x \quad \text{on } \{x_d = 0\}
			\end{equation} and suppose that, for some $r>0$, we have $$\begin{cases}
			\|p\|_{L^\infty(B_1)}\le Cr,\\   
               |\nu'-e_d|\le C r^{1+\alpha},\\
                  |\tau'-\tau|\le Cr^{1+\alpha},\\
                  c\le|\tau'|\le C,\\
                  |\xi'-\xi|\le Cr^{1+\beta},\\ 
                  |\xi'|\le r, 
			\end{cases}
			$$ for some universal constants $C>0$ and $c>0$. Then there exists a quadratic polynomial $p'$ such that \begin{equation*}
				\nabla p' \cdot \tau' = \xi'\cdot x \quad \text{on } \{x\cdot \nu' = 0\}.
			\end{equation*} and
			$$\|p-p'\|_{L^\infty(B_1)} \le Cr^{1+\beta}$$
			for some universal constant $C>0$.
		\end{lemma}
		\begin{proof}
			Let $M\in\R^{d\times d}$ such that $p=\frac12Mx\cdot x$. We notice that \eqref{e:boundary_cond_p_W} is equivalent to require that the scalar product between the $j$-th column of $M$ and the vector $\tau$ is equal $\xi_j$, for every $j=1,\ldots,d-1$. We choose $Q\in\R^{d\times d}$ such that
			$$
			Q_{ij} = \omega_{ij} M_{ij}+\sigma_{ij},
			$$
			where the coefficients $\omega_{ij},\sigma_{ij}\in\R$ must be chosen in order to satisfy 
			\begin{equation}
				\label{e:cond_intermedia}
				\nabla \widetilde p \cdot \tau' = \xi'\cdot x \quad \text{on } \{x_d= 0\},
			\end{equation}
			where $\widetilde p := \frac{1}{2} Q x \cdot x$. Precisely, \eqref{e:cond_intermedia} holds if we choose $\omega_{ij}$,
			with $$\omega_{ij}:=\tau_i/\tau'_i\quad\text{for every}\quad i,j=1\ldots d$$ 
			and $$\sigma_{ij}:=\begin{cases}
				(\xi_j'-\xi_j)/{\tau_i'}\quad&\text{if }i=\ell,\\
				0\quad&\text{if }i\not=\ell,
			\end{cases}$$ where $\ell$ is such that $\tau'_\ell\not=0.$
			In particular it holds the estimate
			$$
			\|p-\widetilde{p}\|_{L^\infty(B_1)} \le C (\|p\|_{L^\infty(B_1)}|\tau'-\tau|+|\xi'-\xi|)\le Cr^{1+\beta}.
			$$
			Now let $R\in\R^{d\times d}$ be a rotation such that $R\nu'=e_d$. Then $|R-I|\le C|\nu'-e_d|$.
			Similarly as above, we can construct $\overline p$ such that
			$$
			\nabla\overline p\cdot \left(R\tau'\right)=R\xi'\cdot x \quad\text{on } \{x_d= 0\},
			$$
			with the estimate
			\bea
			\|\overline p-\widetilde p\|_{L^\infty(B_1)}&\le C(\|\widetilde p\|_{L^\infty(B_1)}|\tau'||R-I|+|\xi'||R-I|)\le C(\|\widetilde p\|_{L^\infty(B_1)}|\nu'-e_d|+|\xi'||\nu'-e_d|)\\&\le Cr^{1+\beta}.\eea
			Finally, by defining $p'(x):=\overline p(Rx)$, we have 
			$$\nabla p'(x)\cdot \tau'=\nabla \overline p(Rx)\cdot R\tau'=R\xi'\cdot Rx=\xi'\cdot x\quad\text{for every}\quad x\in\{x\cdot \nu'= 0\}$$
			and
			$$
			\|p'-\overline p\|_{L^\infty(B_1)}\le C\|\overline p\|_{L^\infty(B_1)}|R-I|\le C\|\overline p\|_{L^\infty(B_1)}|\nu'-e_d|\le Cr^{1+\beta},
			$$ 
			which concludes the proof
		\end{proof}
		\begin{proof}[Proof of \cref{t:improvement of flatness}] We can apply the $C^{2,\alpha_0}$ estimates from \cref{prop:c2alphaestimate} to the function $\widetilde u$ defined in \cref{lemma:linearized-problem}, then there is a universal constant $ C_0>0$ such that $$|\widetilde u(x)-\nabla\widetilde u(0)\cdot x-\frac12D^2\widetilde u(0)x\cdot x |\le  C_0|x|^{2+\alpha_0}\quad\text{for every} \quad x\in B_{1/4}^+\cup B_{1/4}'.$$
			We define 
			$$\nu:=\nabla\widetilde u(0)\in\R^d\quad\text{and}\quad W:=D^2\widetilde u(0)\in\mathcal{S}^{d\times d}.$$
			Then, by \cref{corollary:convergence-uniform-hausdorff}, up to multiplying $C_0$ by a universal constant and taking $n$ large enough, we get
			$$x\cdot \nu+\frac12Wx\cdot x- C_0\rho^{2+\alpha_0}\le\widetilde u_n(x)\le x\cdot \nu+\frac12Wx\cdot x+ C_0\rho^{2+\alpha_0}\quad\text{for every} \quad x\in\overline\Omega_{u_n}\cap B_{\rho},$$
			for some $\rho\in(0,1/4)$ to be chosen. We take 
			$$\overline \nu_n:=g_n(0,e_d)e_d+r_n^{1+\alpha}\nu\quad\text{and}\quad \overline p_n(x):=p_n(x)+r_n^{1+\alpha}\frac12Wx\cdot x,$$
			then the flatness can be rewritten as
			$$x\cdot \overline\nu_n+\overline p_n(x)-C_0r_n^{1+\alpha}\rho^{2+\alpha_0}\le u_n(x)\le x\cdot \overline\nu_n+\overline p_n(x)+C_0r_n^{1+\alpha}\rho^{2+\alpha_0}\quad\text{for every } \quad x\in\overline\Omega_{u_n}\cap B_{\rho}.$$ 
			
             \smallskip
			\noindent\textit{Step 1: changing the vector.}
			We need to modify the vector $\overline \nu_n$ since does not belong to $\mathbb{S}^{d-1}$ and so we define $\nu_n\in\mathbb{S}^{d-1}$ as the unit vector in direction $\overline \nu_n$. Then 
			\be\label{1est}\begin{aligned}\nu_n-e_d&=\frac{g_n(0,e_d)e_d+r_n^{1+\alpha}\nu-e_d\sqrt{(g_n(0,e_d)+r_n^{1+\alpha}\nu_d)^2+(r_n^{1+\alpha}\nu')^2}}{\sqrt{(g_n(0,e_d)+r_n^{1+\alpha}\nu_d)^2+(r_n^{1+\alpha}\nu')^2}}\\&=r_n^{1+\alpha}\frac{\nu-\nu_de_d}{g_n(0,e_d)}+O(r_n^2)
			\end{aligned}\ee then, by \ref{hyp:h1}, we have \be\label{eq:eq10}
			|\nu_n-e_d|\le Cr_n^{1+\alpha}\ee
			for some universal constant $C>0$. 
			Moreover, by \ref{eq:vicinanza-g1} we have
			\bea g_n(0,\nu_n)-\sqrt{(g_n(0,e_d)+r_n^{1+\alpha}\nu_d)^2+(r_n^{1+\alpha}\nu')^2}&=g_n(0,\nu_n)-g_n(0,e_d)-r_n^{1+\alpha}\nu_d+O(r_n^2)\\&=\nabla_\theta g_n(0,e_d)\cdot (\nu_n-e_d)-r_n^{1+\alpha}\nu_d+O(r_n^{1+\beta})\\&=r_n^{1+\alpha}\frac{\nabla_\theta g_n(0,e_d)}{g_n(0,e_d)}\cdot (\nu-\nu_de_d)-r_n^{1+\alpha}\nu_d\\&\qquad+O(r_n^{1+\beta})
			\eea where in the last inequality we used \eqref{1est}.
			Let $\tau(\cdot,\cdot)$ as in \eqref{def:tau}, then, by \cref{lemma:linearized-problem}, we have $\nu\cdot \tau=0$ and so $\nu\cdot \tau(e_d,g_n)$ converges to $0$ as $n\to+\infty$, by definition (see \eqref{tau(e_d)} and \eqref{def2}).
            Then we obtain
			\bea g_n(0,\nu_n)-\sqrt{(g_n(0,e_d)+r_n^{1+\alpha}\nu_d)^2+(r_n^{1+\alpha}\nu')^2}&=r_n^{1+\alpha}\nu\cdot \tau(e_d,g_n)+O(r_n^2)=r_n^{1+\alpha}o(1),\eea as $n\to+\infty$,
			therefore
			$$|g_n(0,\nu_n)\nu_n-\overline \nu_n|=|\nu_n||g_n(0,\nu_n)-|\overline \nu_n||
			=r_n^{1+\alpha}o(1)\quad\text{as } n\to+\infty.$$
			It follows that
			$$g(0,\nu_n)(x\cdot \nu_n)+\overline p_n(x)-C_0r_n^{1+\alpha}\rho^{2+\alpha_0}\le u_n(x)\le g(0,\nu_n)(x\cdot \nu_n)+\overline p_n(x)+C_0r_n^{1+\alpha}\rho^{2+\alpha_0}$$ for every $x\in\overline\Omega_{u_n}\cap B_{\rho},$ up to multiplying $C_0$ by a universal constant and taking $n$ large enough.
			
			\smallskip
			\textit{Step 2: changing the polynomial.}
			Now we need to modify the polynomials $\overline p_n$, since do not belong to the class of polynomials $\mathcal{P}(\nu_n,\F_n, f_n, g_n)$.
			First of all, we claim that for every $n$, we can find a polynomial $\widetilde p_n$ such that 
			\be\label{eq:firstcondition}\nabla \widetilde p_n(x)\cdot \tau(\nu_n,g_n)=\nabla_xg_n(0,\nu_n)\cdot x\quad \text{for every}\quad x\in\{x\cdot\nu_n=0\}\ee and \be\label{eq:secondcondition}\|\widetilde p_n-\overline p_n\|_{L^\infty(B_1)}=r_n^{1+\alpha}o(1)\quad\text{as }n\to+\infty,\ee where $\tau(\cdot,\cdot)$ is as in \eqref{def:tau}.
			Indeed we observe that, taking $w(x):=\frac12Wx\cdot x$, by \cref{lemma:linearized-problem}, we have that $$\nabla w\cdot \tau=0\quad\text{on }\{x_d=0\},$$  
            then, we can find $W_n\in\mathcal{S}^{d\times d}$ which converges to $W$ as $n\to+\infty$ and such that $w_n(x):=\frac12W_nx\cdot x$ satisfies $$\nabla w_n\cdot \tau(e_d,g_n)=0\quad\text{on }\{x_d=0\},$$ by \eqref{tau(e_d)} and \eqref{def2}.
			Taking $p_n+r_n^{1+\alpha}w_n$ and applying \cref{lemma:stima_vicino_p_p'}, which can be applied by \ref{eq:vicinanza-g1}, \ref{eq:vicinanza-p} and \eqref{eq:eq10}, there exists a quadratic polynomial $\widetilde p_n$ satisfying \eqref{eq:firstcondition} such that
			\begin{equation*}
				\begin{aligned}\|\widetilde p_n-p_n-r_n^{1+\alpha}w_n\|_{L^\infty(B_1)}\le Cr_n^{1+\beta}.
				\end{aligned}
			\end{equation*} 
			This implies that	\be\label{eq:ultima30}\begin{aligned} \|\widetilde p_n-\overline p_n\|_{L^\infty(B_1)}&\le\|\widetilde p_n-p_n-r_n^{1+\alpha}w_n\|_{L^\infty(B_1)}+r_n^{1+\alpha}\|w_n-w\|_{L^\infty(B_1)}
				\le r_n^{1+\alpha}o(1),\end{aligned}\ee as $n\to+\infty$. Then \eqref{eq:secondcondition} holds and the claim is proved.
			
			We take $\tau_n^\perp\in\mathbb{S}^{d-1}$ as a unit vector orthogonal to $\tau(\nu_n,g_n)$ and we take $$p_n'(x):=\widetilde p_n(x)+\frac12(t_n\tau_n^\perp\otimes \tau_n^\perp)x\cdot x$$ for $t_n\in\R$ so that
			$$\F_n(D^2p_n',g_n(0,\nu_n)\nu_n,0)=f_n(0)$$ and \bea t_n&=O(f_n(0)-\F_n(D^2\widetilde p_n,g_n(0,\nu_n)\nu_n,0)
            \\&=O(f_n(0)-\F_n(D^2\widetilde p_n,g_n(0,e_d)e_d,0))+O(r|g(0,\nu_n)\nu_n-g(0,e_d)e_d|),\eea by \ref{eq:vicinanza-F1}. We can choose $t_n$ with this property, indeed we have that $$t_n\mathcal{M}^-(\tau_n^\perp\otimes \tau_n^\perp)-r|g_n(0,\nu_n)\nu_n-g_n(0,e_d)e_d|\le \F(D^2p_n',g_n(0,\nu_n)\nu_n,0)-\F(D^2\widetilde p_n,g_n(0,e_d)e_d,0)$$ and $$\F(D^2p_n',g_n(0,\nu_n)\nu_n,0)-\F(D^2\widetilde p_n,g_n(0,e_d)e_d,0)\le t_n\mathcal{M}^+(\tau_n^\perp\otimes \tau_n^\perp)+r|g_n(0,\nu_n)\nu_n-g_n(0,e_d)e_d|$$ by \eqref{eq:pucci} and \ref{eq:vicinanza-F1}, where $$\mathcal{M}^+(\tau_n^\perp\otimes \tau_n^\perp)=\Lambda\quad\text{and}\quad \mathcal{M}^-(\tau_n^\perp\otimes \tau_n^\perp)=\lambda. $$
			By definition of $p_n'$, we also have
			$$\nabla p_n'(x)\cdot\tau(\nu_n,g_n)=\nabla_x g_n(0,\nu_n)\cdot x\quad\text{for every}\quad x\in \{x\cdot\nu_n=0\},$$ then $p_n'\in\mathcal{P}(\nu_n,\F_n,f_n,g_n).$
			We claim that 
            \be\label{t_n}t_n=r_n^{1+\alpha}o(1)\quad\text{as } n\to+\infty.\ee
            Observe that \be\label{eq:thirdcondition}|g_n(0,\nu_n)\nu_n-g_n(0,e_d)e_d|\le C r_n^{1+\alpha},\ee by \ref{eq:vicinanza-g1} and \eqref{eq:eq10}. Let $\widetilde \F_n$ is as in \eqref{ftilde}, then
			\be\label{usefultilde} r_n^{1+\alpha}\widetilde\F_n(M,0,0)=\F_n(D^2p_n+r_n^{1+\alpha}M,g_n(0,e_d)e_d,0)-\F(D^2p_n,g_n(0,e_d)e_d,0).\ee
			Then, by \eqref{eq:pucci}, \ref{eq:vicinanza-F1}, \eqref{eq:ultima30}, \eqref{eq:thirdcondition} and \eqref{usefultilde}, we obtain \begin{align*}\F_n(D^2\widetilde p_n,g_n(0,e_d)e_d,0)&=\F_n(D^2 \overline p_n,g_n(0,e_d)e_d,0)+O(||D^2\widetilde p_n-D^2 \overline p_n||_{L^\infty(B_1)})
				\\&=\F_n(D^2 p_n+r_n^{1+\alpha}W,g_n(0,e_d)e_d,0)+r_n^{1+\alpha}o(1)
				\\&= \F_n(D^2p_n,g_n(0,e_d)e_d,0)+r_n^{1+\alpha}\widetilde\F_n(W,0,0)+r_n^{1+\alpha}o(1)
				\\&=\F_n(D^2p_n,g_n(0,\nu_n)\nu_n,0)+r_n^{1+\alpha}\widetilde\F_n(W,0,0)+r_n^{1+\alpha}o(1)
				\\&=
				f_n(0)+r_n^{1+\alpha}o(1),\end{align*} as $n\to+\infty$, where we used  and that $\widetilde \F_n(W,0,0)$ converges to $\widetilde \F_\infty(W)=0$ by \cref{lemma:linearized-problem}, \eqref{Finfty} and \eqref{def2}.
			The claim \eqref{t_n} is proved.
			Thus we get 
			\begin{align*}\|  p_n'-\overline p_n\|_{L^\infty(B_1)}&\le\|  p_n'-\widetilde p_n\|_{L^\infty(B_1)}+\|  \widetilde p_n-\overline p_n\|_{L^\infty(B_1)}=r_n^{1+\alpha}o(1)\quad\text{as } n\to+\infty
			\end{align*} by \eqref{eq:ultima30} and \eqref{t_n}.
			We also have
			\begin{align*}\|p_n'- p_n\|_{L^\infty(B_1)}&\le 
				\|p_n'- \overline p_n\|_{L^\infty(B_1)}+\|\overline p_n- p_n\|_{L^\infty(B_1)}\le Cr_n^{1+\alpha}
			\end{align*} for some universal constant $ C>0$.
			Therefore, up to multiplying $C_0$ by a universal constant and taking $n$ large enough, we have that $$g_n(0,\nu_n')(x\cdot \nu'_n)+ p_n'(x)-C_0r_n^{1+\alpha}\rho^{2+\alpha_0}\le u_n(x)\le g_n(0,\nu_n')(x\cdot \nu'_n)+ p_n'(x)+C_0r_n^{1+\alpha}\rho^{2+\alpha_0}$$ for every $x\in \overline\Omega_{u_n}\cap B_{\rho}$.
			Finally, if we choose $\rho>0$ such that $C_0\rho^{2+\alpha_0}\le \rho^{2+\alpha}$, we get a contradiction with the initial assumptions.
		\end{proof}
		
		\section{Flatness implies \texorpdfstring{$C^{2,\alpha}$}{C2a}}\label{section6}
		In this section we prove \cref{t:flatness-implies}, namely, the $C^{2,\alpha}$ regularity of flat free boundaries.
		We iterate the quadratic improvement of flatness in \cref{t:improvement of flatness} to prove the following proposition, which is the rate of convergence for the Taylor expansion for a solution $u$ up to order two.
		\begin{proposition}[Second order Taylor expansion with rate of convergence]\label{corollary:rate-of-convergence} There are universal constants $\eps_0>0$ and $C>0$ such that the following holds. Suppose that $\F\in\EE$, $f\in C^{0,\beta}(B_1)$, $g\in C^{1,\beta}(B_1,\mathbb{S}^{d-1})$ satisfy the hypotheses \ref{hyp:h1}, \ref{hyp:h2}, \ref{hyp:h3}. Let $u:B_1\to \R$ be a non-negative continuous viscosity solution to \eqref{def:def-viscosity-solution}, such that $$(g(0,\nu)(x\cdot \nu)-\eps_0)_+\le u(x)\le (g(0,\nu)(x\cdot \nu)+ \eps_0)_+\quad\text{for every}\quad x\in B_1,$$ for some unit vector $\nu\in\mathbb{S}^{d-1}$. 			
			Then, for every $x_0\in \partial\Omega_u\cap B_{1/2}$, there is a unit vector $\nu_{x_0}\in\mathbb{S}^{d-1}$ and a quadratic polynomial $p_{x_0}$ such that
			$$\|u(x_0+x)-g(0,\nu_{x_0})(x\cdot \nu_{x_0})-p_{x_0}(x)\|_{L^\infty(B_{r}\cap\Omega_u)}\le Cr^{2+\alpha}\quad\text{for every}\quad r\in(0,1/2).$$
		\end{proposition}	
		\noindent In order to prove the $C^{2,\alpha}$ estimate in \cref{corollary:rate-of-convergence}, we need a preliminary lemma.
		\begin{lemma}\label{lemma:rate}
	There is a universal constant $\delta>0$ such that the following holds.
        Suppose that $\F\in\EE$, $f\in C^{0,\beta}(B_1)$, $g\in C^{1,\beta}(B_1,\mathbb{S}^{d-1})$ satisfy the hypotheses \ref{hyp:h1}, \ref{hyp:h2}, \ref{hyp:h3}.
            Let $u:B_1\to\R$ be a non-negative continuous viscosity solution to \eqref{def:def-viscosity-solution}, then, for every $x_0\in \partial\Omega_u\cap B_{1/2}$, we have
			\be\label{eq:eq1}[\F_{x_0,\delta}]_{\text{Lip}(\R^d)}\le r_0^{2},\quad [\F_{x_0,\delta}]_{C^{0,\beta}(B_1)}\le r_0^{2}, \quad [f_{x_0,\delta}]_{C^{0,\beta}(B_1)}\le r_0^{2}, \quad [g_{x_0,\delta}]_{C^{1,\beta}(B_1)}\le  r_0^{2},\ee and
			\be\label{eq:eq2}
			\|f_{x_0,\delta}\|_{L^\infty(B_1)}\le r_0^{2}, \quad \|\nabla_x g_{x_0,\delta}\|_{L^\infty(B_1\times \mathbb{S}^{d-1})}\le r_0^{2},\quad 
			\ee where $\F_{x_0,\delta}$, $f_{x_0,\delta}$ and $g_{x_0,\delta}$ are the rescalings in \eqref{def:rescaling-operator} and $r_0>0$ is as in \cref{t:improvement of flatness}.
			Moreover, if $\rho_n:=\delta\rho^n$ and $r_n:=r_0\rho^n$, where $\rho>0$ is as in \cref{t:improvement of flatness}, then the sequence $\F_n:=\F_{x_0,\rho_n}$, $f_n:=f_{x_0,\rho_n}$, $g_n:=g_{x_0,\rho_n}$ satisfy the hypotheses \ref{eq:vicinanza-F1}, \ref{eq:vicinanza-f}, \ref{eq:vicinanza-g1} in \cref{t:improvement of flatness}.
		\end{lemma}
		\begin{proof} 
        The existence of a universal $\delta>0$ such that \eqref{eq:eq1} and \eqref{eq:eq2} hold follows by \eqref{eq:pucci}, \ref{hyp:h3} and by definitions of $\F_{x_0,\delta}$, $f_{x_0,\delta}$, $g_{x_0,\delta}$.
        Using \eqref{eq:eq1} and \eqref{eq:eq2}, we have the following estimates.
			\begin{itemize}
				\item The rescalings $\F_n$ satisfy \ref{hyp:h2} and \ref{eq:vicinanza-F1}, indeed
                \begin{align*}\|\F_n(M,\xi,\cdot)-\F_n(M,\eta,\cdot)\|_{L^\infty(B_1)}&= \rho^n\left\|\F_{x_0,\delta}\left(\frac{1}{\rho^n}M,\xi,\cdot\right)-\F_{x_0,\delta}\left(\frac{1}{\rho^n}M,\eta,\cdot\right)\right\|_{L^\infty(B_1)}\\&\le \rho^n[\F_{x_0,\delta}]_{\text{Lip}(\R^d)}|\xi-\eta|\le r_n|\xi-\eta|.
				\end{align*} and, by \ref{hyp:h3}, \begin{align*}\|\F_n(M,\xi,x)-\F_n(M,\xi,0)\|_{L^\infty(B_1)}&= \rho^n\biggl\|\F_{x_0,\delta}\left(\frac{1}{\rho^n}M,\xi,\rho^nx\right)-\F_{x_0,\delta}\left(\frac{1}{\rho^n}M,\xi,0\right)\biggl\|_{L^\infty(B_1)}\\&\le[\F_{x_0,\delta}]_{C^{0,\beta}(B_1)}\rho^n\left(1+|\xi|+\frac{1}{\rho^n}|M|\right)\rho^{n\beta}\\&\le r_n^{\beta}(r_n+r_n|\xi|+|M|).
				\end{align*}

                \smallskip
                \item The rescalings $f_n$ satisfy \ref{eq:vicinanza-f}, indeed
				$$\|f_n(x)-f_n(0)\|_{L^\infty(B_1)}= \rho^n\|f_{x_0,\delta}(\rho^nx)-f_{x_0,\delta}(0)\|_{L^\infty(B_1)}\le  [f_{x_0,\delta}]_{C^{0,\beta}(B_1)}\rho^{n(1+\beta)} \le r_n^{1+\beta}.$$
				
                \smallskip
                \item The rescalings $g_n$ satisfy \ref{hyp:h1} and \ref{eq:vicinanza-g1}, indeed \bea\|g_n(x&,\nu_1)-g_n(0,\nu_0)-\nabla_xg_n(0,\nu_0)\cdot x-\nabla_\theta g_n(0,\nu_0)\cdot (\nu_1-\nu_0)\|_{L^\infty(B_1)}\\&=\|g_{x_0,\delta}(\rho^nx,\nu_1)-g_{x_0,\delta}(0,\nu_0)-\nabla_xg_{x_0,\delta}(0,\nu_0)\cdot x-\nabla_\theta g_{x_0,\delta}(0,\nu_0)\cdot (\nu_1-\nu_0)\|_{L^\infty(B_1)}\\&\le [g_{x_0,\delta}]_{C^{1,\beta}(B_1)}(\rho^{n(1+\beta)}+|\nu_1-\nu_0|^{1+\beta})\\&\le r_n^{1+\beta}+C|\nu_1-\nu_0|^{1+\beta},
				\eea
	            \begin{align*}|\nabla_xg_n(0,\nu_1)-\nabla_xg_n(0,\nu_0)|&=\rho^n|\nabla_xg_{x_0,\delta}(0,\nu_1)-\nabla_xg_{x_0,\delta}(0,\nu_0)|\le \rho^n[g_{x_0,\delta}]_{C^{1,\beta}(B_1)}|\nu_1-\nu_0|^\beta\\&\le r_n|\nu_1-\nu_0|^\beta,\end{align*} 
				\begin{align*}|\nabla_\theta g_n(0,\nu_1)-\nabla_\theta g_n(0,\nu_0)|&=|\nabla_\theta g_{x_0,\delta}(0,\nu_1)-\nabla_\theta g_{x_0,\delta}(0,\nu_0)|\le [g_{x_0,\delta}]_{C^{1,\beta}(B_1)}|\nu_1-\nu_0|^\beta,\end{align*} 
			$$\|\nabla g_n(0,\cdot)\|_{L^\infty(\mathbb{S}^{d-1})}=\rho^n \|\nabla_x g_{x_0,\delta}(0,\cdot)\|_{L^\infty(\mathbb{S}^{d-1})}\le r_n.$$ \qedhere
			\end{itemize}
		\end{proof}
		Now we are ready to prove \cref{corollary:rate-of-convergence}.
		\begin{proof}[Proof of \cref{corollary:rate-of-convergence}] Without loss of generality, we can suppose that $\nu=e_d$. Let $x_0\in \partial\Omega_u\cap B_{1/2}$,
			we take $\eps_0:=\delta r_0^{2}$ where $r_0>0$ is as in \cref{t:improvement of flatness} and $\delta>0$ as in \cref{lemma:rate}.
			We consider $u_{x_0,\delta}(x):=\frac{u(x_0+\delta x)}{\delta}$, then, by \ref{hyp:h1}, we get
			$$(g(0,e_d)x_d-r_0^{1+\beta})_+\le u_{x_0,\delta}(x)\le (g(0,e_d)x_d+r_0^{1+\beta})_+\quad\text{for every}\quad x\in B_1,$$ up to choose $r_0$ small enough.
			Now we take a quadratic polynomial $p_0(x)$ in the class $\mathcal{P}(e_d,\F_{x_0,\delta},f_{x_0,\delta},g_{x_0,\delta})$, from \eqref{def:polynomial} (the existence of $p_0$ can be easily proved as in \Cref{section-contradiction}). In particular, using again \ref{hyp:h1}, we obtain			\be\label{est:p}\|p_0\|_{L^\infty(B_1)}\le C(\|f_{x_0,\delta}\|_{L^\infty(B_1)}+\|\nabla_x g_{x_0,\delta}\|_{L^\infty(B_1\times \mathbb{S}^{d-1})})\le Cr_0^{2}.\ee Therefore
			\bea\label{eq:cond}(g(0,e_d)x_d+p_0(x)- r_0^{1+\alpha})_+ \le u_{x_0,\delta}(x)\le (g(0,e_d)x_d+p_0(x)+r_0^{1+\alpha})_+\quad\text{for every}\quad x\in B_1, \eea up to choose $r_0$ small enough.
			
			We claim that the hypotheses of \cref{t:improvement of flatness} are satisfied by the rescaled function $u_n(x)=\frac{u(x_0+\rho_nx)}{\rho_n}$, where $\rho_n:=\delta\rho^n$ and $\rho>0$ is as in \cref{t:improvement of flatness}.
			If the claim is true, then by \cref{t:improvement of flatness} we know that there are unit vectors $\nu_n\in\mathbb{S}^{d-1}$ and quadratic polynomials $p_n$ such that $$(x\cdot\nu_n+ p_n(x)-r_n^{1+\alpha})_+\le u_{n}(x)\le (x\cdot\nu_n+ p_n(x)+r_n^{1+\alpha})_+\quad\text{for every}\quad x\in B_1,$$ where $r_n:=r_0\rho^n$. 
			
			Now we prove the claim. First observe that the functions $u_n$ are solutions to \eqref{def:def-viscosity-solution} with operator $\F_n:=\F_{x_0,\rho_n}$, right-hand side $f_n:=f_{x_0,\rho_n}$, free boundary condition $g_n:=g_{x_0,\rho_n}$ as in \eqref{def:rescaling-operator}, which satisfy the hypotheses \ref{eq:vicinanza-F1}, \ref{eq:vicinanza-f}, \ref{eq:vicinanza-g1}, by \cref{lemma:rate}.
			Moreover notice that if $p_n$ is the polynomial at step $n$ and $p_n'$ is the modified polynomial as in \cref{t:improvement of flatness}, then $p_{n+1}:=\rho p_n'$. In particular, by \eqref{est:p}, we have that
            \begingroup\allowdisplaybreaks\begin{align*}\|p_{n+1}\|_{L^\infty(B_1)}&=\rho\|p'_{n}\|_{L^\infty(B_1)}\le\rho\|p_{n}\|_{L^\infty(B_1)}+\overline C \rho r_{n}^{1+\alpha}=\rho^2\|p'_{n-1}\|_{L^\infty(B_1)}+\overline C \rho r_{n}^{1+\alpha}\\&\le \rho^2\|p_{n-1}\|_{L^\infty(B_1)}+\overline C\rho^2 r_{n-1}^{1+\alpha}+\overline C\rho r_{n}^{1+\alpha}
				\\&\le\ldots\le \rho^{n+1}\|p_0\|_{L^\infty(B_1)}+\overline C\sum_{j=0}^{n}\rho^{n+1-j}r_j^{1+\alpha}
				\\&=\rho^{n+1}\|p_0\|_{L^\infty(B_1)}+\overline Cr_0^{1+\alpha}\sum_{j=0}^{n}\rho^{n+1-j}(\rho^j)^{1+\alpha}
				\\&=\rho^{n+1}\|p_0\|_{L^\infty(B_1)}+\overline Cr_0^{1+\alpha}\rho^{n+1}\sum_{j=0}^n \rho^{j\alpha}
				\\&\le r_0\rho^{n+1}=r_{n+1},
			\end{align*}\endgroup up to choose $r_0$ small enough.
			Since $$p_n\in\mathcal{P}(\nu_n,\F_n,f_n,g_n),$$ the claim is proved.

			Now observe that, by \eqref{eq:stime}, we have $$|\nu_{n+1}-\nu_{n}|\le \overline Cr_n^{1+\alpha}$$ and \begin{align*}\left\|\frac{p_{n+1}}{r_{n+1}}-\frac{p_{n}}{r_{n}}\right\|_{L^\infty(B_1)}&=\frac{1}{r_0\rho^n}\left\|\frac{p_{n+1}}{\rho}-p_{n}\right\|_{L^\infty(B_1)}=\frac{1}{r_0\rho^n}\left\|p'_{n}-p_{n}\right\|_{L^\infty(B_1)}
				\le \overline C\frac{r_n^{1+\alpha}}{r_0\rho^n}=\overline Cr_n^\alpha\end{align*} 
			It follows that $\nu_n\to \nu_{x_0}$ and $\frac{p_n}{r_n}\to  \widetilde p_{x_0}$ for some unit vector $\nu_{x_0}\in\mathbb{S}^{d-1}$ and for some quadratic polynomial $ \widetilde p_{x_0}$.
			The flatness can be rewrite as $$\|u_n-g(0,\nu_{x_0})(x\cdot\nu_{x_0})-r_n \widetilde p_{x_0}\|_{L^\infty(B_1\cap \Omega_u)}\le Cr_n^{1+\alpha},$$ which implies the thesis after an interpolation with the radii $\rho_n$.
		\end{proof}
		\begin{proof}[Proof of \cref{t:flatness-implies}] Once the $C^{2,\alpha}$ rate of convergence was proved in \cref{corollary:rate-of-convergence}, the proof of the $C^{2,\alpha}$ regularity of the free boundary $\partial\Omega_u\cap B_1$ follows by standard arguments (see e.g. \cite{Velichkov:RegularityOnePhaseFreeBd}).
		\end{proof}

        \section{Higher regularity}\label{section8}
		In this section we prove the higher regularity of the free boundary $\partial\Omega_u\cap B_1$ in \cref{corollary:cinfty}. Once the $C^{2,\alpha}$ regularity of $u$ was established, the higher regularity result in \cref{corollary:cinfty} follows applying a hodograph map as done in \cite[Theorem 2]{KinderlehrerNirenberg1977:AnalyticFreeBd}.
	\begin{proof}[Proof of \cref{corollary:cinfty}]
         By \cref{t:flatness-implies}, we know that $u\in C^{2,\alpha}(\overline \Omega_u\cap B_{1/2})$. Given $\F\in C^{k,\beta}(\mathcal{S}^{d\times d}\times \R^d\times B_1)$, $f\in C^{k,\beta}(B_1)$, $g\in C^{k+1,\beta}(B_1,\mathbb{S}^{d-1})$ (resp. analytic), by a covering argument, it is sufficient to prove that, for every $x_0\in\partial\Omega_u\cap B_{1/2}$, there exists a radius $\delta>0$ such that the free boundary $\partial\Omega_u\cap B_1$ is of class $C^{k+2,\beta}$ (resp. analytic) in $B_\delta(x_0)$.
            Without loss of generality, we can suppose that $x_0=0$ and $\nabla u(0)=g(0,e_d)e_d$. Then, we can choose $\delta>0$ small enough such that the hodograph map $$\Phi:\overline\Omega_u\cap B_\delta\to \R^d\cap\{y_d\ge0\},\quad \Phi(x',x_d):=(x',u(x',x_d)),\quad $$ is bijective from $\overline\Omega_u\cap B_\delta$ onto $\{y_d\ge0\}\cap U$, where $U$ is a neighborhood $0$, and maps the free boundary $\partial\Omega_u\cap B_1$ into $\{y_d=0\}\cap U $.
            Then the partial Legendre transformation 
            $$\Phi^{-1}:\{y_d\ge0\}\cap U\to \overline\Omega_u\cap B_\delta,\quad \Phi^{-1}(y',y_d):=(y',w(y',y_d)),\quad $$ is well-defined
			and the free boundary $\partial\Omega_u\cap B_\delta$ is the graph of $$\{y_d\ge0\}\cap U\ni y'\mapsto\Phi^{-1}(y',0)= (y',w(y',0))$$ Then, in order to prove a regularity result for the free boundary $\partial\Omega_u\cap B_\delta$, it is sufficient to prove the same regularity result for the function $w$.

For every $x\in \Omega_u\cap B_\delta$, we can derive the identity $w(x',u(x))=x_d$ to get
\bea\label{der1}\partial_{y_d}w(x',u(x))\partial_{x_d}u(x)=1\eea and \bea\label{der2}\partial_{y_i}w(x',u(x))+\partial_{y_d}w(x',u(x))\partial_{x_i}u(x)=0\quad\text{if }i<d.\eea 
Deriving again, for every $x\in \Omega_u\cap B_\delta$, we have
\bea\label{der3}\partial_{y_dy_d}w\,(\partial_{x_d}u)^2+\partial_{x_dx_d}u\,\partial_{y_d}w=0,\eea \bea\label{der4}\partial_{y_dy_i}w\,\partial_{x_d}u+\partial_{y_dy_d}w\,\partial_{x_i}u\,\partial_{x_d}u+\partial_{y_d}w\,\partial_{x_ix_d}u=0\quad\text{if }i<d\eea and \bea\label{der5}\partial_{y_iy_j}w+\partial_{y_dy_i}w\,\partial_{y_j}u+\partial_{y_jy_d}w\,\partial_{x_i}u+\partial_{y_dy_d}w\,\partial_{x_i}u\,\partial_{x_j}u+\partial_{x_ix_j}u\,\partial_{y_d}w=0\quad\text{if }i,j<d,\eea where the function $w$ is computed in $(y',y_d)=(x',u(x))=\Phi(x)$.
Using the above identities, we get
$$\begin{cases}
\partial_{x_i}u=-\frac{\partial_{y_i}w}{\partial_{y_d}w}\quad&\text{if }i<d,\\
\partial_{x_d}u=\frac{1}{\partial_{y_d}w},\\
\partial_{x_dx_d}u=-\frac{\partial_{y_dy_d}w}{(\partial_{y_d}w)^3},\\
\partial_{x_dx_i}u=-\frac{\partial_{y_dy_i}w}{(\partial_{y_d}w)^2}+\frac{\partial_{y_i}w\partial_{y_dy_d}w}{(\partial_{y_d}w)^3}\quad&\text{if }i<d,\\
\partial_{x_ix_j}u=-\frac{\partial_{y_iy_j}w}{\partial_{y_d}w}+\frac{\partial_{y_i}w\partial_{y_dy_j}w}{(\partial_{y_d}w)^2}+\frac{\partial_{y_j}w\partial_{y_dy_i}w}{(\partial_{y_d}w)^2}-\frac{\partial_{y_i}w\partial_{y_j}w\partial_{y_dy_d}w}{(\partial_{y_d}w)^3}\quad&\text{if }i,j<d.
\end{cases}$$ 
This means that the first and second derivatives of $u$ in $x$ can be explicitly expressed in terms of first and second derivatives of $w$ in $y$.
In particular, if $x=(y',w(y))$, there are explicit functions $\Psi_1:\mathcal{S}^{d\times d}\times \R^d\to\mathcal{S}^{d\times d}$ and $\Psi_2:\R^d \to \R^d$ such that $$D^2u(x)=\Psi_1(D^2w(y),\nabla w(y))\quad \text{and}\quad \nabla u(x)=\Psi_2(\nabla w(y)).$$ 
We define
$$\widetilde \F(M,\xi,z,y):=\F(\Psi_1(M,\xi),\Psi_2(\xi),(y',z))-f(y',z),$$ then $$\widetilde \F(D^2w,\nabla w,w,y)=\F(D^2u,\nabla u,x)-f(x)=0 \quad\text{in }\{y_d>0\}\cap U,$$
    Moreover, as can be easily seen, the operator $\widetilde \F$ is also elliptic with respect to $w$ (see \cite[Lemma 3.1]{KinderlehrerNirenberg1977:AnalyticFreeBd}).
    
            For what concerns the boundary condition, we define
            $$G(x,\xi):=|\xi|-g\left(x,\frac{\xi}{|\xi|}\right).$$ Since $\nabla u(0)=g(0,e_d)e_d$, an explicit computation shows that \be\label{cond-fond}\partial_{\xi_d}G(0,\nabla u(0))=1.\ee We also set
            $$\widetilde G(y',z,\xi):=G((y',z),\Psi_2(\xi))=G\left((y',z),\left(-\frac{\xi_1}{\xi_d},\ldots,-\frac{\xi_{d-1}}{\xi_{d}},\frac{1}{\xi_d}\right)\right)$$
            then we observe that $$\widetilde G(y',w,\nabla w)=G(x,\nabla u(x))=0\quad\text{on } \{y_d=0\}\cap U.$$ 
            Moreover, by \eqref{cond-fond} and since $\nabla u(0)=g(0,e_d)e_d$, we have that \be\label{cond-fond1}\begin{aligned}\partial_{\xi_d}\widetilde G(y,w(y),\nabla w(y))&=\nabla_{\xi} G(x,\nabla u(x))\cdot\left(\frac{\partial_{y_1}w}{(\partial_{y_d}w)^2},\ldots,\frac{\partial_{y_{d-1}}w}{(\partial_{y_d}w)^2},-\frac{1}{(\partial_{y_d}w)^2}\right)\not=0
            \end{aligned}
            \ee
for every $y\in \{y_d=0\}\cap U$, for $\delta$ small enough.

           In order to use the regularity results in \cite{adn,morrey-multiple-integrals} for uniform elliptic problem, we need to check that the elliptic system \begin{equation}\label{eq:problem-compl}\begin{cases}
			\widetilde \F(D^2w,\nabla w,w,y)=0 \quad&\text{in } \{y_d>0\}\cap U,\\
			\widetilde G(y',w,\nabla w)=0 \quad&\text{on } \{y_d=0\}\cap U,\\
		\end{cases}
	\end{equation}
    satisfies the \textit{complementing condition} on the boundary $\{y_d=0\}$ in the sense of \cite[Section 7]{adn}. It is well known (see e.g. \cite{kns}) that this condition is equivalent to show the following fact. For every $\xi'\in\R^{d-1}\setminus\{0\}$ and $y_0\in \{y_d=0\}\cap U$, there are no nontrivial solutions to the elliptic system  \begin{equation}\label{ellipticproblem}\begin{cases}
			\LL v=0 \quad&\text{in } \{y_d>0\},\\
			\mathcal{B}v=0 \quad&\text{on } \{y_d=0\},\\
		\end{cases}
	\end{equation}
    where $$\LL v:=\sum_{j,k=1}^d\frac{\partial \widetilde \F}{\partial{M_{jk}}}(D^2w(y_0),\nabla w(y_0),w(y_0),y_0)\partial_{y_iy_j}v=\sum_{j,k=1}^d a_{ij}\partial_{y_jy_k}v$$ and $$\mathcal{B}v:=\sum_{j=1}^d\frac{\partial \widetilde G}{\partial{\xi_{j}}}(y_0,w(y_0),\nabla w(y_0))\partial_{y_j}v=\sum_{j=1}^db_j\partial_{y_j}v$$
            of the type \bea v(y',y_d):=e^{i \xi'\cdot y'}v_0(y_d),\quad v_0:\R\to\R,\eea which decays at $0$ as $y_d\to+\infty$.
        Given $v$ as above, the interior condition of \eqref{ellipticproblem} gives
$$a_{dd}v''_0(y_d)+i\sum_{j=1}^{d-1}(a_{jd}+a_{dj})\xi_j'v_0'(y_d)-\sum_{j,k=1}^{d-1}a_{jk}\xi_j'\xi_k'v_0(y_d)=0,$$ then $$av''_0(y_d)+ibv'_0(y_d)+cv_0(y_d)=0,$$ for some $a,b,c\in\R$.
Let $\mu_{1,2}:=\frac{-ib\pm\sqrt{-b^2-4ac}}{2a}$, then we have two cases.
\begin{itemize}
    \item If $\mu_1\not=\mu_2$, then $$v_0(y_d)=C_1e^{\mu_1y_d}+C_2e^{\mu_2y_d}.$$ Since $v_0$ vanishes as $y_d\to+\infty$, if $C_j\not=0$ for $j=1,2$, then $\text{Re}(\mu_j)\le 0$. For the same reason, it cannot be happen that $\text{Re}(\mu_1)=\text{Re}(\mu_2)=0$.
    The free boundary condition of \eqref{ellipticproblem} gives $$b_d(C_1\mu_1+C_2\mu_2)+i(C_1+C_2)\sum_{j=1}^{d-1}b_j\xi'_j=0,$$ then $$C_1\text{Re}(\mu_1)+C_2\text{Re}(\mu_2)=0,$$ since $b_d\not=0$, by \eqref{cond-fond1}. By definition of $\mu_j$, we know that $\text{Re}(\mu_1)=-\text{Re}(\mu_2)$, then either $C_1=C_2$ or $\text{Re}(\mu_1)=0$. The second possibility is impossible, then $C_1=C_2=0$.
\item If $\mu_1=\mu_2=-\frac{ib}{2a}$, then $$v_0(y_d)=C_1e^{\mu_1y_d}+C_2y_de^{\mu_1y_d}.$$ Then $C_1=C_2=0$, since $v_0$ vanishes as $y_d\to+\infty$.
\end{itemize}
This proves the claim, i.e.~the elliptic problem \eqref{eq:problem-compl} satisfies the complementing condition.

As already observed, $u\in C^{2,\alpha}(\overline \Omega_u\cap B_\delta)$, then $w\in C^{2,\alpha}(\{y_d\ge0\}\cap U)$.
Moreover, since $\F\in C^{k,\beta}$, $f\in C^{k,\beta}$ and $g\in C^{k+1,\beta}$, we deduce that $\widetilde \F\in C^{k,\beta}$ and $\widetilde G\in C^{k+1,\beta}$.
            Then we can apply \cite[Theorem 11.1]{adn} to get that $w$, and thus $\partial \Omega_u\cap B_\delta$, is of class $C^{k+2,\beta}$.
			Similarly, in the case when $\F,$ $f,$ $g,$ are  analytic, we can apply the results in \cite[Section 6.7]{morrey-multiple-integrals}
            to get that $w$, and thus $\partial \Omega_u\cap B_\delta$, is analytic.
		\end{proof}
		\printbibliography
	\end{document}